\documentclass[11pt]{amsart}

\usepackage{amsmath, amsfonts, amsthm, amssymb} 
\usepackage{enumitem}
\usepackage{multirow}
\usepackage{hyperref}
\usepackage{cite}
\usepackage[all]{xy}

\hypersetup{
  colorlinks   = true,
  urlcolor     = blue,
  linkcolor    = red,
  citecolor    = blue 
}

\theoremstyle{plain}
\newtheorem{theorem}{Theorem}[section]
\newtheorem{lemma}[theorem]{Lemma}
\newtheorem*{shorttheoremstatement}{Theorem \ref{shorttheorem}}

\newtheorem{corollary}[theorem]{Corollary}

\newtheorem{proposition}[theorem]{Proposition}

\theoremstyle{definition}
\newtheorem{definition}[theorem]{Definition}
\newtheorem{example}[theorem]{Example}
\newtheorem*{problem}{Problem}

\theoremstyle{remark}
\newtheorem*{remark}{Remark}
\newtheorem*{question}{Question}
\newtheorem*{assumption}{Standing Assumption}

\newcommand{\R}{\ensuremath{\mathbb{R}}}
\newcommand{\Z}{\ensuremath{\mathbb{Z}}}
\newcommand{\Q}{\ensuremath{\mathbb{Q}}}

\newcommand{\C}{\ensuremath{\mathbb{C}}}
\newcommand{\D}{\ensuremath{\mathbf{D}}}
\newcommand{\field}{\ensuremath{K}}

\newcommand{\U}{\ensuremath{\mathbf{U}}}
\newcommand{\F}{\ensuremath{\mathbf{F}}}
\newcommand{\bfa}{\ensuremath{\mathbf{A}}}
\newcommand{\A}{\bfa}
\newcommand{\bfh}{\ensuremath{\mathbf{H}}}
\newcommand{\bfl}{\ensuremath{\mathbf{L}}}
\newcommand{\bft}{\ensuremath{\mathbf{T}}}
\newcommand{\bfs}{\ensuremath{\mathbf{S}}}
\newcommand{\bfn}{\ensuremath{\mathbf{N}}}
\newcommand{\bfg}{\ensuremath{\mathbf{G}}}

\newcommand{\mchn}{\ensuremath{\mathcal{H}^{2n+1}}}

\newcommand{\Id}{\ensuremath{\mathrm{Id}}}
\newcommand{\nn}{\ensuremath{\mathfrak{n}}}

\newcommand{\ahf}{\ensuremath{\mathcal{A}_{\bfh \mid
      \F}}}
\newcommand{\ahu}{\ensuremath{\mathcal{A}_{\bfh \mid
      \U}}}
\newcommand{\as}{\ensuremath{\mathcal{A}_{\bft}^1}}
\newcommand{\incl}{\ensuremath{i}}
\newcommand{\into}{\hookrightarrow}

\providecommand{\cyc}[1]{\left\langle#1\right\rangle}

\DeclareMathOperator{\Mod}{Mod}
\DeclareMathOperator{\Aut}{Aut}
\DeclareMathOperator{\Out}{Out}
\DeclareMathOperator{\Comm}{Comm}
\DeclareMathOperator{\Fitt}{Fitt}

\DeclareMathOperator{\rank}{rank}
\DeclareMathOperator{\Inn}{Inn}
\DeclareMathOperator{\GL}{GL}
\DeclareMathOperator{\SL}{SL}
\DeclareMathOperator{\PSL}{PSL}
\DeclareMathOperator{\PGL}{PGL}
\DeclareMathOperator{\SO}{SO}
\DeclareMathOperator{\SU}{SU}
\DeclareMathOperator{\Sp}{Sp}
\DeclareMathOperator{\GSp}{GSp}
\DeclareMathOperator{\Rad}{Rad}

\DeclareMathOperator{\Nil}{Nil}
\DeclareMathOperator{\rrank}{rank_\R}

\newcommand\restr[2]{{
  \left.\kern-\nulldelimiterspace 
  #1 
  \vphantom{\big|} 
  \right|_{#2} 
  }}

\newcommand\suchthat{ \mathrel{}\middle| \mathrel{}}

\renewcommand{\bold}[1]{\medskip \noindent {\bf #1 }\nopagebreak}

\newcommand{\italics}[1]{\medskip \noindent {\it #1 }\nopagebreak}

\begin{document}

\title{Abstract Commensurators of Lattices in Lie Groups}
\author{Daniel Studenmund} 
\date{\today}
\maketitle

\begin{abstract}
  Let $\Gamma$ be a lattice in a simply-connected solvable Lie
  group. We construct a $\Q$-defined algebraic group $\mathcal{A}$
  such that the abstract commensurator of $\Gamma$ is isomorphic to
  $\mathcal{A}(\Q)$ and $\Aut(\Gamma)$ is commensurable with
  $\mathcal{A}(\Z)$. Our proof uses the algebraic hull construction,
  due to Mostow, to define an algebraic group $\bfh$ so that
  commensurations of $\Gamma$ extend to $\Q$-defined automorphisms of
  $\bfh$. We prove an analogous result for lattices in connected
  linear Lie groups whose semisimple quotient satisfies superrigidity.
\end{abstract}

\tableofcontents

\section{Introduction}

Given a group $\Gamma$, its {\em abstract commensurator}
$\Comm(\Gamma)$ is the set of equivalence classes of isomorphisms
between finite index subgroups of $\Gamma$, where two isomorphisms are
equivalent if they agree on a finite index subgroup of
$\Gamma$. Elements of $\Comm(\Gamma)$ are called {\em commensurations}
of $\Gamma$. The abstract commensurator forms a group under
composition.

The computation of $\Comm(\Gamma)$ is a fundamental
problem. Commensurations play an important role in the study of
rigidity, see e.g. \cite{zimmer} and \cite{margulis}. Commensurations
also arise in classification problems in geometry and topology,
e.g. \cite{neumannreid}, \cite{farbweinberger1},
\cite{farbweinberger2}, \cite{leiningerlongreid}, and \cite{avramidi}.

The structure of $\Comm(\Gamma)$ is often much richer than that of
$\Aut(\Gamma)$. For example, $\Aut(\Z^n) \cong \GL_n(\Z)$ while
$\Comm( \Z^n ) \cong \GL_n(\Q)$. There are a few notable exceptions,
which include the cases that $\Gamma$ is a higher genus mapping class
group, that $\Gamma = \Out(F_n)$ for $n \geq 4$, or that $\Gamma$ is a
nonarithmetic lattice in a semisimple Lie group without compact
factors and not locally isomorphic
to $\PSL_2(\R)$. In these cases, $\Comm(\Gamma)$ is virtually
isomorphic to $\Gamma$; see \cite{ivanov}, \cite{farbhandel}, and
\cite{margulis}, respectively.

This paper is motivated by the following problem.
\begin{problem}
  Let $G$ be a (connected, linear, real) Lie group and let
  $\Gamma \leq G$ be a lattice. Compute $\Comm(\Gamma)$.
\end{problem}

\begin{assumption} Unless otherwise noted, in this paper every Lie
  group is assumed to be real and connected, and to admit a faithful
  continuous linear representation. In particular, semisimple Lie
  groups have finite center.
\end{assumption}

 Every Lie group $G$ satisfies a short exact sequence
\[
1 \to \Rad(G) \to G \to G^{ss} \to 1,
\]
where $\Rad(G)$ is the maximal connected solvable normal Lie subgroup
of $G$, and $G^{ss}$ is semisimple. The study of Lie groups therefore
roughly splits into three pieces: one for solvable groups, one for
semisimple groups, and a final piece to combine the previous two. Our
computation of $\Comm(\Gamma)$ follows this outline.

\bold{Semisimple $G$:} Suppose $G$ is a connected semisimple Lie
group, not locally isomorphic to $\SL_2(\R)$, and $\Gamma \leq G$ is
an irreducible lattice. Then the computation of $\Comm(\Gamma)$ is a
result of work by Borel, Mostow, Prasad, and Margulis. Recall that the
{\em relative commensurator} of $\Gamma$ in $G$ is defined as
\[
\Comm_G(\Gamma) = \left\{ g\in G \suchthat \Gamma \cap g
  \Gamma g^{-1} \text{ is of finite index in }\Gamma \text{ and }g\Gamma
  g^{-1} \right \}.
\]
\begin{itemize}
  [
  leftmargin=\parindent
  ,itemsep=3pt
  ]
\item If $\Gamma$ is abstractly commensurable to $\bfg(\Z)$ for some
  $\Q$-defined, adjoint semisimple algebraic group $\bfg$ with no
  $\Q$-defined normal subgroup $\bfn$ such that $\bfn(\R)$ is compact,
  then $\Comm_\bfg(\Gamma) = \bfg(\Q)$ \cite{boreldensity}. (Such a
  lattice $\Gamma$ is called {\em arithmetic}.) For example, if
  $\Gamma = \PSL_n(\Z)$ for $n\geq 2$, then $\Gamma$ is abstractly
  commensurable with the group $\bfg(\Z)$, where $\bfg$ is the
  semisimple algebraic group $\bfg = \PGL_n$, and so
  $\Comm_\bfg(\Gamma) \cong \bfg(\Q)$; see \S\ref{pslnsection} for
  details.

\item A major theorem of Margulis \cite{margulis} says that $\Gamma$
  is arithmetic if and only if $[ \Comm_G(\Gamma) : \Gamma] = \infty$,
  which occurs if and only if $\Comm_G(\Gamma)$ is dense in $G$.
   
\item If $G$ has no center and no compact factors, then every
  commensuration of $\Gamma$ extends to an automorphism of $G$ by
  Mostow--Prasad--Margulis rigidity \cite{mostowrigidity}.

\item The inner automorphisms of a semisimple real Lie group are
  finite index in the automorphism group. Therefore
  \[
  \Comm(\Gamma) \doteq 
  \begin{cases} 
    \bfg(\Q) & \text{if } \Gamma \text{ is arithmetic} \\
    \Gamma & \text{if } \Gamma \text{ is non-arithmetic}, 
  \end{cases}
  \]
  where $H \doteq K$ if and only if $H$ and $K$ are abstractly
  commensurable, i.e. contain isomorphic finite index subgroups. See
  Theorem \ref{semicomm} for a more precise statement.
\end{itemize}

\begin{remark} In the case $G = \PSL_2(\R)$, every lattice is either
  virtually free or virtually the fundamental group of a closed
  surface. In either case, the abstract commensurator is not
  linear; see Proposition \ref{psl2comm}. The abstract commensurator
  of a surface group has been studied in \cite{odden} and
  \cite{biswasnag}, and may be described as a certain subgroup of the
  mapping class group of the universal 2-dimensional hyperbolic
  solenoid.
\end{remark}

\bold{Solvable $G$:} Suppose $G$ is a connected, simply-connected
solvable real Lie group and $\Gamma \leq G$ is a lattice. In contrast
with the semisimple case, $\Aut(\Gamma)$ is not typically abstractly
commensurable with $\Gamma$. On the other hand, the fact that
$\Aut(\Gamma)$ is commensurable with the $\Z$-points of a $\Q$-defined
algebraic group holds for both arithmetic lattices in higher rank
semisimple groups and lattices in simply-connected solvable groups, at
least on passage to a subgroup of finite index in $\Gamma$; see
\cite[1.12]{bauesgrunewald} and \cite[Ch8]{segalbook}.

In the case that $\Gamma$ is a lattice in a simply-connected nilpotent
group, arithmeticity of $\Aut(\Gamma)$ is a classical result of
Baumslag and Auslander. Merzljakov \cite{merz} showed that
$\Aut(\Gamma)$ embeds in some $\GL_n(\Z)$ for any polycyclic group
$\Gamma$, and this was extended to virtually polycyclic groups by
Wehrfritz \cite{wehrfritz}. For more history and a detailed discussion
of arithmeticity results see
\cite{bauesgrunewald}, whose Theorem 1.3 provides a deeper statement
on the structure of $\Aut(\Gamma)$ for virtually polycyclic $\Gamma$.

This similarity between arithmetic semisimple lattices and solvable
lattices is reflected in their abstract commensurators. For example,
consider $G = \R^n$ and $\Gamma = \Z^n$. Then $\Aut(\Z^n) = \GL_n(\Z)$
is arithmetic in the $\Q$-defined real algebraic group $\Aut(\R^n) =
\GL_n(\R)$, and $\Comm(\Gamma) = \GL_n(\Q)$. Our first main theorem
extends this to lattices in arbitrary simply-connected solvable
groups, following techniques of \cite{bauesgrunewald}.
\begin{theorem} \label{introsolv} Let $\Gamma$ be a lattice in a
  connected, simply-connected solvable Lie group $G$. Then there is
  some $\Q$-defined algebraic group $\mathcal{A}_\Gamma$ such that
  \[
  \Comm(\Gamma) \cong \mathcal{A}_\Gamma(\Q)
  \]
  and the image of $\Aut(\Gamma)$ in $\mathcal{A}_\Gamma(\Q)$ is
  commensurable with $\mathcal{A}_\Gamma(\Z)$.
\end{theorem}
\begin{remark}
  If $G$ is `sufficiently nice' then $\mathcal{A}_\Gamma(\R) =
  \Aut(G)$. This is proved in Theorem \ref{nilpotentcomm} in the case
  that $G$ is nilpotent. See Proposition \ref{thesisadded1} for a more
  general result.
\end{remark}
\begin{remark}
  Any virtually polycyclic group contains a subgroup of finite index
  that embeds as a lattice in a connected, simply-connected solvable
  Lie group. Therefore Theorem \ref{introsolv} describes
  $\Comm(\Gamma)$ for any virtually polycyclic group $\Gamma$.
\end{remark}

A fundamental difficulty in dealing with lattices in solvable groups
is lack of rigidity; automorphisms of a lattice may not extend to
automorphisms of its ambient Lie group, even virtually. There are a
number of results addressing this to some extent, most notably
\cite{witte}. Instead of applying results providing rigidity in the
ambient Lie group, our proof of Theorem \ref{introsolv} uses methods
developed by Baues and Grunewald in \cite{bauesgrunewald}, following
work of Grunewald and Platonov \cite{grunewaldplatonov,
  grunewaldplatonovrigidity}. 

Our proof utilizes the {\em virtual algebraic hull}, a connected
solvable $\Q$-defined algebraic group $\bfh$ in which $\Gamma$
virtually embeds as a Zariski-dense subgroup. The construction of the
virtual algebraic hull is due to Mostow \cite{mostowrep}. (See
\cite[\S4]{raghunathan} for an alternate construction.) There is a
natural map
\[
\xi : \Comm(\Gamma) \to \Aut(\bfh)
\]
such that $\xi([\phi])$ is $\Q$-defined for each $[\phi] \in
\Comm(\Gamma)$. The automorphism group $\Aut(\bfh)$ naturally has the
structure of a $\Q$-defined algebraic group, and we set
$\mathcal{A}_\Gamma$ equal to the Zariski-closure of $\xi(
\Comm(\Gamma) )$ in $\Aut(\bfh)$. Note that our map $\xi$ extends the
map $\Aut(\tilde \Gamma) \to \Aut(\bfh)$ defined in
\cite[\S4.1]{bauesgrunewald} for some subgroup $\tilde\Gamma \leq
\Gamma$ of finite index.

\begin{remark}
  Baues extends Mostow's algebraic hull construction to certain
  virtually polycyclic groups $\Gamma$ in \cite{bauesinfrasolv}, and
  this hull is applied in \cite{bauesgrunewald} to describe
  $\Aut(\Gamma)$ and $\Out(\Gamma)$. Though our proof of Theorem
  \ref{introsolv} is heavily based on the techniques in
  \cite{bauesgrunewald}, we use only the identity component of the
  algebraic hull. This is because $\Comm(\Gamma)$ only depends on
  $\Gamma$ up to commensurability.
\end{remark}

\begin{remark}
  Though the group $\mathcal{A}_\Gamma$ of Theorem \ref{introsolv} is
  defined abstractly, a finite index subgroup of $\Comm(\Gamma)$ can
  be understood fairly concretely. There is a unique maximal
  normal nilpotent subgroup $\Fitt(\Gamma) \leq \Gamma$. Let $\F$
  denote the Zariski-closure of $\Fitt(\Gamma)$ in $\bfh$. Define
  $\Comm_{\bfh \mid \F}(\Gamma)$ to be the group of commensurations
  trivial on $\Gamma / \Fitt(\Gamma)$. By rigidity of tori,
  $\Comm_{\bfh \mid \F}$ is of finite index in $\Comm(\Gamma)$. The
  group $\Comm_{\bfh \mid \F}$ decomposes as the product of the group
  of commensurations arising from conjugation by elements of $\F(\Q)$
  and the group of commensurations fixing a maximal $\Q$-defined torus
  $\bft \leq \bfh$. See \S\ref{commhfsection} and \S\ref{solvsection}
  for details.
\end{remark}

\bold{General $G$:} When $G$ is not necessarily either semisimple or
solvable, we prove:

\begin{theorem}\label{shorttheorem}
  Suppose $G$ is a connected, linear Lie group with connected,
  simply-connected solvable radical. Suppose $\Gamma \leq G$ is a
  lattice with the property that there is no surjection $\phi: G \to
  H$ to any group $H$ locally isomorphic to any $\SO(1,n)$ or
  $\SU(1,n)$ so that $\phi(\Gamma)$ is a lattice in $H$.  Then
  \begin{enumerate}
  \item $\Gamma$ virtually embeds in the group of $\Q$-points of a
    $\Q$-defined algebraic group $\bfg$ with Zariski-dense image so
    that every commensuration $[\phi] \in \Comm(\Gamma)$ induces a
    unique $\Q$-defined automorphism of $\bfg$ virtually extending
    $\phi$.
  \item There is a $\Q$-defined algebraic group
    $\mathcal{B}$ so that
    \[
    \Comm(\Gamma) \cong \mathcal{B}(\Q)
    \]
    and the image of $\Aut(\Gamma)$ in $\mathcal{B}$ is commensurable
    with $\mathcal{B}(\Z)$.
  \end{enumerate}
\end{theorem}

The group $\bfg$ of Theorem \ref{shorttheorem} is, roughly speaking,
constructed as the semidirect product of the virtual algebraic hull
$\bfh$ of the ``solvable part'' of $\Gamma$ and a semisimple group
$\bfs$ such that $\bfs(\Z)$ is commensurable with the ``semisimple
part'' of $\Gamma$.  The technical work comes first in making this
precise, and second in constructing an action of $\bfs$ on $\bfh$
compatible with the group structure of $\Gamma$.

The hypothesis that $\Gamma$ does not surject to a lattice in either
$\SO(1,n)$ or $\SU(1,n)$ is used to apply the superrigidity results of
Margulis and Corlette, which are used to extend commensurations of
$\Gamma$ to automorphisms of $\bfg$. In the case that $\Gamma$
surjects to a non-superrigid lattice, our construction may fail to
produce a candidate group $\bfg$. Even in the presence of such a
candidate group $\bfg$, commensurations do not generally extend to
automorphisms of $\bfg$. Additional commensurations arise from the
nontriviality of $H^1(\Gamma, \Q)$; see the remark at the end of
\S\ref{gensection}.

\begin{remark}
  If $\bfa$ is a $\Q$-defined algebraic group, then there is a natural
  map $\Xi : \Aut_\Q( \bfa ) \to \Comm(\bfa(\Z))$. If $\bfa$ is
  unipotent, or if $\bfa$ is $\Q$-simple, semisimple, and such that
  $\bfa(\R)$ is not compact and has no factor isogenous to
  $\PSL_2(\R)$, then $\Xi$ is injective because $\bfa(\Z)$ is
  Zariski-dense in $\bfa$, and $\Xi$ is surjective because $\bfa(\Z)$
  is strongly rigid in $\bfa$ by results of Malcev and
  Mostow--Prasad--Margulis (see Theorems \ref{nilpotentcomm} and
  \ref{semisimplecomm}). See \cite{grunewaldplatonov} for analogous
  results in the case that $\bfa$ is solvable.

  The difficulty in proving our results comes from the fact that
  lattices in solvable Lie groups need not be commensurable with the
  $\Z$-points of any algebraic group; see \cite{segalbook} for an
  example. When $\Gamma$ is a lattice in a simply-connected solvable
  group, the algebraic hull construction provides an algebraic group
  $\bfh$ so that $\Gamma$ virtually embeds in $\bfh(\Z)$ as a
  Zariski-dense subgroup, but in general the image of this embedding
  may be of infinite index. Despite this, automorphisms of $\bfh$
  extending commensurations of $\Gamma$ may be understood in terms of
  the algebraic structure of $\bfh$.
\end{remark}


\bold{Outline:} We review basic results in the theory of linear
algebraic groups in \S\ref{prelimsection}. We define and review basic
properties of the abstract commensurator in \S\ref{commsection},
including definitions of commensuristic and strongly commensuristic
subgroups.

In \S\ref{nilsection} we prove Theorem \ref{introsolv} for nilpotent
$G$ using classical rigidity of nilpotent lattices. In
\S\ref{hullsection}, we review the basic theory of polycyclic groups
and the definition of the algebraic hull. Our exposition largely
follows \cite{bauesgrunewald}. We define the unipotent shadow and
discuss the algebraic structure of $\Aut(\bfh)$. In
\S\ref{solvsection} we prove Theorem \ref{introsolv}.

In \S\ref{semisection} we review results on commensurations of
lattices in semisimple Lie groups, which are due primarily to Borel,
Mostow, Prasad, and Margulis. In \S\ref{gensection} we combine the
solvable and semisimple cases to prove Theorem \ref{shorttheorem}.


\bold{Acknowledgements:} I am pleased to acknowledge helpful
conversations with Matt Emerton, Wouter van Limbeek, Madhav Nori, John
Sun, Preston Wake, Alex Wright, and Kevin Wortman. I am deeply
grateful to Dave Morris for many helpful conversations and
correspondences about lattices in Lie groups. Benson Farb and Wouter
van Limbeek provided helpful comments on early drafts of this
paper. The anonymous referee provided many comments and suggestions
that significantly improved the exposition of this paper. Above all, I
am immensely thankful to Benson Farb for setting me on my feet and
providing me with support and encouragement as I learned to walk.

\section{Notation and preliminaries} \label{prelimsection}

If $g,h$ are elements of a group, their commutator is written $[g,h] =
g h g^{-1} h^{-1}$. A group $\Gamma$ {\em virtually} has a property
$P$ if there is a finite index subgroup of $\Gamma$ with $P$.  In
particular, if $\Gamma \leq G$, say that a homomorphism $\phi : \Gamma
\to H$ virtually extends to a homomorphism $\Phi : G \to H$ if there
is a finite index subgroup $\Gamma_0 \leq \Gamma$ so that
$\restr{\phi}{\Gamma_0} = \restr{\Phi}{\Gamma_0}$.

\subsection{Algebraic groups} We use the basic theory of linear
algebraic groups. A good general reference is \cite{borelbook}. Our
preliminaries overlap with those in \cite{bauesgrunewald}.

Let $\field \subseteq \C$ be a subfield. A {\em linear algebraic
  group} $\A$ is a subgroup of $\GL_n(\C)$ for some natural number $n$
that is closed in the Zariski topology. An algebraic group $\A$ is
{\em $\field$-defined} if it is closed in the Zariski topology with
closed subsets those defined by polynomials with coefficients in a
subfield $\field$ of $\C$. A $\field$-defined algebraic group is called a
$\field$-group. A $\field$-group is {\em $\field$-simple} if it has no
connected normal $\field$-defined subgroup, and {\em absolutely
  simple} if it has no connected normal subgroup defined over
$\C$. (Such groups are sometimes called ``almost $\field$-simple'' or
``absolutely almost simple,'' respectively.)

If $R$ is a subring of $\C$, then define $\A(R) = \A \cap \GL_n(R)
\subseteq \GL_n(\C)$. If $V$ is a complex vector space with a fixed
basis, then $V(R)$ denotes the collection of $R$-linear combinations
of basis vectors. Every algebraic group has finitely many
Zariski-connected components. The connected component of the identity
$\A^0$ is a finite index subgroup of $\A$.

\begin{proposition}[cf. {\cite[1.3]{borelbook}}] \label{closuredefinition}
  If $\A$ is $\field$-defined and $\Gamma \leq \A(\field)$ is a
  subgroup, then the Zariski-closure of $\Gamma$ is a $\field$-defined
  subgroup.
\end{proposition}
\begin{proposition}[cf. {\cite[18.3]{borelbook}}] \label{kpointsdense}
  If $\A$ is a connected $\field$-defined algebraic group, then
  $\A(\field)$ is Zariski-dense in $\A$.
\end{proposition}

A {\em homomorphism of algebraic groups} is a group homomorphism that
is also a morphism of the underlying affine algebraic varieties. If both
varieties are $\field$-defined and the variety morphism is defined
over $\field$, then we say that the homomorphism of algebraic groups
is {\em $\field$-defined}. A {\em $\field$-defined isomorphism} is a
$\field$-defined morphism of algebraic groups with an inverse that is
also $\field$-defined. Let $\Aut(\A)$ denote the group of
automorphisms of $\A$ as an algebraic group, and $\Aut_\field(\A)$
denote the group of $\field$-defined automorphisms of $\A$.

Quotients and semi-direct products of $\field$-defined algebraic
groups exist:
\begin{lemma}[cf. {\cite[6.8]{borelbook}}] \label{Qquotient} Suppose
  $G$ is a $\field$-defined algebraic group and $H \leq G$ is a
  normal, closed, $\field$-defined subgroup. Then $G / H$ is a
  $\field$-defined algebraic group, and the quotient map $\pi : G \to
  G/H$ is $\field$-defined.
\end{lemma}
\begin{lemma}[cf. {\cite[1.11]{borelbook}}] \label{Qproduct} Suppose
  $G$ and $H$ are $\field$-defined algebraic groups. Suppose $G$ acts
  on $H$, and the action map $\alpha : G \times H \to H$ is
  $\field$-defined. Then the semi-direct product $H\rtimes G$
  naturally has the structure of a $\field$-defined algebraic group.
\end{lemma}

A {\em torus} is an algebraic group isomorphic to $(\C^*)^n$ for some
$n$. Because the automorphism group of a torus is discrete, we have:
\begin{lemma}[cf. {\cite[8.10]{borelbook}}] \label{toririgidity} Let
  $\bft$ be any torus and $\A$ any algebraic group acting on
  $\bft$ by homomorphisms, so that the map $\A \times \bft \to
  \bft$ is a morphism of varieties. Then $\A^0$ acts trivially on
  $\bft$.
\end{lemma}

Let $\A$ be a $\field$-defined algebraic group. The {\em unipotent
  radical} $\U_\A$ of $\A$ is the unique maximal closed unipotent
normal subgroup of $\A$. The {\em solvable radical} $\Rad(\A)$ of $\A$
is unique maximal connected closed solvable normal subgroup of
$\A$. Both $\U_\A$ and $\Rad(\A)$ are $\field$-defined subgroups of
$\A$. Say $\A$ is {\em reductive} if $\U_\A$ is trivial, and {\em
  semisimple} if $\Rad(\A)$ is trivial. A {\em Levi subgroup} is a
connected reductive subgroup $\bfl \leq \A$ so that $\A = \U_\A
\rtimes \bfl$.
\begin{theorem}[Mostow, see {\cite[Theorem 2.3]{platonovrapinchuk}}]
  \label{levimostow}
  For any $\field$-defined algebraic group $\A$, there is a
  $\field$-defined Levi subgroup $\bfl$. Moreover, any reductive
  $\field$-defined subgroup is conjugate by an element of
  $\U_\A(\field)$ into $\bfl$.
\end{theorem}

The following summarizes some standard results concerning solvable
algebraic groups.
\begin{proposition}[cf. {\cite[10.6]{borelbook}}] \label{solvablealgebraic}
  Let $\bfh$ be a $\Q$-defined connected solvable algebraic
  group. Then:
  \begin{enumerate}
    \renewcommand{\theenumi}{SG\arabic{enumi}}
  \item $\U_\bfh$ consists of all unipotent elements of $\bfh$.
  \item $[\bfh, \bfh] \subseteq \U_\bfh$.
  \item There is a $\Q$-defined maximal torus $\bft \leq \bfh$.
  \item Any two maximal $\Q$-defined tori are conjugate by an element of
    $[\bfh, \bfh](\Q)$.
  \item If $\bft$ is a $\Q$-defined maximal torus, then $\bfh$ is a
    semidirect product $\bfh = \U_\bfh \rtimes \bft$.
  \item If $\D$ is the centralizer of a maximal torus and $\F \leq
    \U_\bfh$ is any normal subgroup containing $[\bfh, \bfh]$, then $\bfh =
    \F \cdot \D$.
  \end{enumerate}
\end{proposition}

\subsection{Semisimple Lie and algebraic groups} A general reference
for the theory of semisimple algebraic groups used here is
\cite[Chapter 1]{margulis}.

If $\bfa$ is an $\R$-defined algebraic group, then $\bfa(\R)$ is a
real Lie group with finitely many connected components. We always
consider $\bfa(\R)$ with its topology as a Lie group. In particular,
$\bfa(\R)^0$ denotes the connected component of the identity in the
Lie group topology. Every connected semisimple Lie group with trivial
center is of the form $\bfs(\R)^0$ for some $\Q$-defined semisimple
algebraic group $\bfs$; for proof see \cite[3.1.6]{zimmer}.

An {\em isogeny} of algebraic groups is a surjective morphism with
finite kernel. An isogeny is {\em central} if its kernel is central. A
connected semisimple algebraic group $\bfs$ is {\em simply-connected}
if every central isogeny $\Phi : \bfs' \to \bfs$ is an
isomorphism. For every connected $\field$-defined semisimple algebraic
group $\bfs$, there is a unique simply-connected $\field$-defined
semisimple algebraic group $\tilde \bfs$ and central $\field$-defined
isogeny $p : \tilde \bfs \to \bfs$. Every simply-connected semisimple
$\field$-group decomposes uniquely into a product of $K$-simple
simply-connected $\field$-groups.

\begin{proposition}[cf. {\cite[I.2.6.5]{margulis}}] \label{repextension}
  Suppose $\bfa$ is an $\R$-defined algebraic group, and $\bfs$ is a
  simply-connected semisimple $\R$-defined algebraic group. Let $\rho
  : \bfs(\R)^0 \to \bfa(\R)$ be a continuous representation. Then
  $\rho$ extends to an $\R$-defined morphism $\tilde \rho : \bfs \to
  \bfa$.
\end{proposition}

A $\Q$-defined semisimple algebraic group $\bfs$ is {\em without
  $\Q$-compact factors} if there is no nontrivial $\Q$-defined
connected normal subgroup $\bfn \leq \bfs$ such that $\bfn(\R)$ is
compact. (This terminology is not standard.) 

\begin{theorem}[Borel Density Theorem
  {\cite{boreldensity}}] \label{boreldensity} Suppose $\bfs$ is a
  connected, $\Q$-defined semisimple algebraic group without $\Q$-compact
  factors. Then $\bfs(\Z)$ is Zariski-dense in $\bfs$.
\end{theorem}

\begin{definition}
  Let $\bfs$ be an $\R$-defined semisimple algebraic group. The {\em
    real rank} of $\bfs$, denoted $\rrank(\bfs)$, is the maximal
  dimension of an abelian $\R$-defined subgroup diagonalizable over
  $\R$. If $S$ is a connected semisimple Lie group with finite center,
  define $\rrank(S)$ to be the real rank of the $\Q$-defined algebraic
  group $\bfs$ satisfying $\bfs(\R)^0 = S / Z(S)$.
\end{definition}

Our results use strong rigidity of Mostow, Prasad, and Margulis, and
superrigidity results of Margulis, Corlette, and Gromov--Schoen. The
following statement is an immediate corollary of
\cite[2.6]{grunewaldplatonovrigidity}.

\begin{theorem}[cf. \cite{grunewaldplatonovrigidity}] 
  \label{semisimplerigidity}
  Suppose $\bfs_1$ and $\bfs_2$ are connected, simply-connected,
  $\Q$-defined, $\Q$-simple semisimple algebraic groups with
  $\rrank(\bfs_1) > 0$ and $\rrank(\bfs_2) > 0$. Suppose $\Gamma_1$
  and $\Gamma_2$ are finite index subgroups of $\bfs_1(\Z)$ and
  $\bfs_2(\Z)$, respectively. Assume that $\bfs_1(\R)^0$ has no simple
  factor locally isomorphic to $\SL_2(\R)$ such that the projection of
  $\Gamma_1 \cap \bfs_1(\R)^0$ into this factor is discrete. Then
  every isomorphism $\Gamma_1 \to \Gamma_2$ virtually extends to a
  $\Q$-defined isomorphism of algebraic groups $\bfs_1 \to \bfs_2$.
\end{theorem}

\section{The abstract commensurator} \label{commsection}

Let $\Gamma$ be an abstract group. In this section we will define the
abstract commensurator $\Comm(\Gamma)$ and review its basic
properties.

A {\em partial automorphism} of $\Gamma$ is an isomorphism $\phi :
\Gamma_1 \to \Gamma_2$ where $\Gamma_1$ and $\Gamma_2$ are finite
index subgroups of $\Gamma$. Two partial automorphisms $\phi$ and
$\phi'$ of $\Gamma$ are {\em equivalent} if there is some finite index
subgroup $\Gamma_3 \leq \Gamma$ so that $\phi$ and $\phi'$ are both
defined on $\Gamma_3$ and $\restr{\phi}{\Gamma_3} =
\restr{\phi'}{\Gamma_3}$. If $\phi: \Gamma_1\to \Gamma_2$ is a partial
automorphism of $\Gamma$, its equivalence class $[\phi]$ is called a
{\em commensuration} of $\Gamma$. There is a natural composition of
commensurations. If $\phi:\Gamma_1 \to \Gamma_2$ and $\phi' :
\Gamma_1' \to \Gamma_2'$ are partial automorphisms of $\Gamma$, then
we define
\[
\left[\phi'\right]\circ \left[ \phi \right] = \left[ \phi' \circ
  \restr{\phi}{\phi^{-1}(\Gamma_2 \cap \Gamma_1')} \right] .
\]
This definition is independent of choice of representatives of
equivalence classes $[\phi]$ and $[\phi']$.
\begin{definition}
  Given a group $\Gamma$, the {\em abstract commensurator}
  $\Comm(\Gamma)$ is the group of commensurations of $\Gamma$ under
  composition. 
\end{definition}
\begin{example}
  $\Comm(\Z^n) \cong \GL_n(\Q)$
\end{example}

Two subgroups $\Delta_1,\Delta_2 \leq \Gamma$ are {\em
  commensurable} if $[\Delta_1 : \Delta_1 \cap \Delta_2] < \infty$
and $[\Delta_2 : \Delta_1 \cap \Delta_2] < \infty$. Define an
equivalence relation on the set of subgroups of $\Gamma$ by $\Delta_1
\sim \Delta_2$ if and only if $\Delta_1$ and $\Delta_2$ are
commensurable. Let $[\Delta]$ denote the equivalence class of a
subgroup $\Delta \leq \Gamma$ under this relation. The abstract
commensurator $\Comm(\Gamma)$ acts on the set of commensurability
classes of subgroups of $\Gamma$ in an obvious way; given a partial
automorphism $\phi : \Gamma_1 \to \Gamma_2$ of $\Gamma$, define
\[
[\phi] \cdot [\Delta] = [\phi(\Delta \cap \Gamma_1)].
\]
Clearly this is independent of choice of representatives $\phi$ and
$\Delta$.
\begin{definition}[Commensuristic subgroup]
  A subgroup $\Delta \leq \Gamma$ is {\em commensuristic} if $[\phi]
  \cdot [\Delta] = [\Delta]$ for every $[\phi] \in \Comm(\Gamma)$.  A
  subgroup $\Lambda \leq \Gamma$ is {\em strongly commensuristic} if,
  for every partial automorphism $\phi: \Gamma_1 \to \Gamma_2$ of
  $\Gamma$,
  \[
  \phi(\Gamma_1 \cap \Lambda) = \Gamma_2 \cap \Lambda. 
  \]
\end{definition}
  
Every strongly commensuristic subgroup is both characteristic
and commensuristic. Neither converse holds.
\begin{example}
  Consider the group 
  \[
  \Gamma = \left \{ \begin{pmatrix} 0 & x & z \\ 0 & 1 & y \\
      0 & 0 & 1 \end{pmatrix} \suchthat x, y \in \Z \text{ and } z\in
    \frac{1}{2} \Z \right \} \leq \GL_3(\Q) . 
  \]
  Note that $\Gamma$ is a lattice in the real Heisenberg group. Denote
  elements of $\Gamma$ by triples $(x,y,z)$ where $x$, $y$, and $z$
  are as above. The center $Z(\Gamma)$ is infinite cyclic, generated
  by $(0,0,\frac{1}{2})$, and contains the commutator subgroup
  $[\Gamma, \Gamma]$ with index 2. By Proposition
  \ref{centralseriescomm}, the center $Z(\Gamma)$ is strongly
  commensuristic and the commutator subgroup $[\Gamma, \Gamma]$ is
  commensuristic. Further, $[\Gamma, \Gamma]$ is evidently
  characteristic.

  Now consider the subgroup $\Gamma_2 \leq \Gamma$ generated by
  $(2,0,0)$, $(0,2,0)$, and $(0,0,2)$. Then the map $\phi: \Gamma \to
  \Gamma_2$ defined by $\phi(x,y,z) = (2x,2y,4z)$ is a partial
  automorphism of $\Gamma$. But $\phi$ takes $[\Gamma, \Gamma]$ to
  $[\Gamma_2, \Gamma_2]$, which is the infinite cyclic group generated
  by $(0,0,4)$. Therefore $[\Gamma, \Gamma]$ is not strongly
  commensuristic. \qed
\end{example}

\begin{question}Let $\Gamma$ be a finitely generated group. Is every
  characteristic subgroup of $\Gamma$ commensuristic? Is every
  commensuristic subgroup of $\Gamma$ commensurable with a
  characteristic subgroup? Is every commensuristic subgroup of
  $\Gamma$ commensurable with a strongly commensuristic subgroup?
\end{question}
 
The notions of `commensuristic' and `strongly
commensuristic' are motivated by the following lemma.
\begin{lemma} \label{inducedmaps}
  If $\Delta \leq \Gamma$ is commensuristic, then restriction
  induces a homomorphism 
  \[
  \Comm(\Gamma) \to \Comm(\Delta). 
  \]
  If $\Delta$ is normal in $\Gamma$ and strongly commensuristic,
  then there is a homomorphism 
  \[
  \Comm( \Gamma ) \to \Comm( \Gamma / \Delta). 
  \]
\end{lemma}
\begin{proof}
  Suppose $\Delta \leq \Gamma$ is commensuristic. Let $\phi: \Gamma_1
  \to \Gamma_2$ be a partial automorphism of $\Gamma$. Then $\phi(
  \Delta \cap \Gamma_1)$ is commensurable with $\Delta$, and so
  $\Delta_1 = \phi^{-1}(\Delta \cap \phi(\Delta \cap \Gamma_1))$ is a
  finite index subgroup of $\Delta$. The restriction of $\phi$ to
  $\Delta_1$ defines a partial automorphism of $\Delta$. Restriction
  clearly respects the equivalence relation on partial automorphisms
  and is compatible with composition, so this determines a
  well-defined homomorphism $\Comm(\Gamma) \to \Comm(\Delta)$.

  Suppose now that $\Delta \leq \Gamma$ is strongly commensuristic and
  normal, and let $\phi: \Gamma_1 \to \Gamma_2$ be a partial
  automorphism of $\Gamma$. Then $\phi$ descends a map $\hat\phi :
  \Gamma_1 \to \Gamma_2 / (\Gamma_2 \cap \Delta)$. Because $\Delta$ is
  strongly commensuristic, the kernel of this map is precisely
  $\Gamma_1 \cap \Delta$. There is then an isomorphism
  \[
  \phi_* : \Gamma_1 / (\Gamma_1 \cap \Delta) \to \Gamma_2 /
  (\Gamma_2 \cap \Delta). 
  \]
  The map $\phi_*$ is a partial automorphism of $\Gamma / \Delta$. If
  $\phi_1$ and $\phi_2$ are equivalent partial automorphisms, then
  $\hat\phi_1$ and $\hat\phi_2$ agree on some finite index subgroup of
  $\Gamma_1$. It follows that $(\phi_1)_*$ and $(\phi_2)_*$ are
  equivalent partial automorphisms of $\Gamma / \Delta$. Therefore
  there is a well-defined map $\Comm(\Gamma) \to \Comm(\Gamma /
  \Delta)$, which is obviously a homomorphism.
\end{proof}
\begin{remark} Lemma \ref{inducedmaps} is inspired by the methods of
  \cite{leiningermargalit}, where the result is applied with $\Gamma =
  B_n$, the braid group on $n$ strands, and $\Delta = Z(B_n)$ as a
  step in the computation of $\Comm(B_n)$ for $n\geq 4$.
\end{remark}
We will often use the following corollaries implicitly in this
paper. Two groups $\Gamma$ and $\Lambda$ are called {\em abstractly
  commensurable}, written $\Gamma \doteq \Lambda$, if there are finite
index subgroups $\Gamma_1 \leq \Gamma$ and $\Lambda_1 \leq \Lambda$
such that $\Gamma_1 \cong \Lambda_1$.

\begin{corollary}
  If $[ \Gamma : \Gamma' ] < \infty$ then $\Comm(\Gamma')
  \cong \Comm(\Gamma)$. \qed
\end{corollary} 
\begin{corollary} \label{abcomm} If $\Gamma \doteq \Lambda$ then
  $\Comm(\Gamma) \cong \Comm(\Lambda)$. \qed
\end{corollary}

There is a weaker notion of equivalence similar to that of abstract
commensurability. Define a relation on groups by $\Gamma_1 \sim
\Gamma_2$ if there is a homomorphism $\phi : \Gamma_1 \to \Gamma_2$
with finite index image and finite kernel. Say that $\Gamma_1$ and
$\Gamma_2$ are {\em commensurable up to finite kernels} if they lie in
the same equivalence class of the equivalence relation generated by
$\sim$. In general, groups which are commensurable up to finite
kernels need not be abstractly commensurable.

Recall that a group $\Gamma$ is {\em residually finite} if the
intersection of all finite index subgroups is trivial. It is a theorem
of Malcev that finitely generated linear groups are residually finite.
The following is an easy exercise that will be used in \S7 and \S8; see
\cite{delaharpe} for proof.
\begin{proposition} \label{resfinite}
  Two residually finite groups are abstractly commensurable if and
  only if they are commensurable up to finite kernels.
\end{proposition}

\section{Commensurations of lattices in nilpotent
  groups} \label{nilsection}

\subsection{Example: the Heisenberg group}
Consider the $(2n+1)$-dimensional Heisenberg group
\[
\mchn = \left \{
  \begin{pmatrix}
    1 & \mathbf{x} & z \\ 
    0 & I_n & \mathbf{y}^t \\
    0 & 0 & 1
  \end{pmatrix}
  \suchthat \mathbf{x}, \mathbf{y} \in \C^n \text{ and } z \in \C
\right \} \leq \GL_{n+2}(\C).
\] 
Then $N = \mchn(\R)$ is a simply-connected, 2--step nilpotent Lie
group in which $\Gamma = \mchn(\Z)$ is a lattice. Let $Z = Z(N)$
denote the center of $N$; note that $Z \cong \R$ and that $N / Z
\cong \R^{2n}$. The group commutator induces a map $N/Z \to Z$ by
$[\mathbf{x}, \mathbf{y}] = \omega( \mathbf{x}, \mathbf{y} )$,
where $\omega$ is the standard symplectic form on $\R^{2n}$.

Suppose $\phi : \Gamma_1 \to \Gamma_2$ is a partial automorphism of
$\Gamma$. We will see that $\phi(\Gamma_1 \cap Z) = \Gamma_2 \cap Z$,
and so $[\phi]$ induces a commensuration $[ \bar \phi ]$ of $\Gamma /
Z(\Gamma) \cong \Z^{2n}$. The induced map $\bar \phi \in \GL_{2n}(\Q)$
has image in the general symplectic group $\GSp_{2n}(\Q)$, defined as
\[
\GSp_{2n}(\Q) = \left\{ A\in \GL_n(\Q) \suchthat \omega( Au, Av ) =
  \alpha \omega(u,v) \text{ for some } \alpha \in \Q^* \right\} .
\]
In fact the induced map $\Theta : \Comm(\Gamma) \to \GSp_{2n}(\Q)$ is
surjective. Each partial automorphism $\phi: \Gamma_1 \to \Gamma_2$
such that $[\phi] \in \ker(\Theta)$ is trivial on $Z$, hence is
determined by an element of $H^1( \pi(\Gamma_1) , \Z )$, where $\pi:
\Gamma \to \Gamma/(Z\cap \Gamma)$ denotes the natural projection. One
can check that
\[
\ker(\Theta) \cong \varprojlim_{[\Gamma : H] < \infty} H^1( \pi(H),
\Z) \cong H^1( \pi(\Gamma), \Q ) \cong \Q^{2n}.
\]

Therefore $\Comm(\Gamma)$ satisfies the short exact sequence
\[
1 \to \Q^{2n} \to \Comm(\Gamma) \to \GSp_{2n}(\Q) \to 1.
\]
The action of $\GSp_{2n}(\Q)$ on $\Q^{2n}$ is the tensor product of
the dual representation with the 1--dimensional representation $\mu :
\GSp_{2n}(\Q) \to \Q^*$ defined by $\omega( Au, Av ) = \mu(A)
\omega(u,v)$.

\subsection{Commensuristic subgroups}
Lattices in simply-connected nilpotent Lie groups provide a source of
examples of commensuristic and strongly commensuristic
subgroups. Recall that the upper central series $\gamma^i(G)$ and
lower central series $\gamma_i(G)$ of a group $G$ are defined
inductively as follows. Let $\gamma^0(G) = 1$. Suppose that
$\gamma^i(G)$ is a normal subgroup of $G$, and let $\pi : G \to G /
\gamma^i(G)$. Define $\gamma^{i+1}(G) = \pi^{-1}( Z( G / \gamma^i(G) )
)$. Now let $\gamma_0(G) = G$. Supposing $\gamma_i(G)$ is defined, set
$\gamma_{i+1}(G) = [G, \gamma_i(G)]$.
\begin{proposition} \label{centralseriescomm}
  Let $\Gamma \leq N$ be a lattice in a simply-connected nilpotent Lie
  group. The upper central series of $\Gamma$ is strongly
  commensuristic in $\Gamma$. The lower central series of $\Gamma$ is
  commensuristic. 
\end{proposition}
\begin{proof}
  A discrete subgroup $\Delta \leq N$ is a lattice in $N$ if and only
  if $\Delta$ is Zariski-dense in some (equivalently, any) faithful
  unipotent representation of $N$ into $\GL_n(\R)$; see
  \cite{raghunathan} for a proof. Using this it is easy to show by
  induction that $\gamma^k(\Gamma) = \Gamma \cap \gamma^k(N)$ for all
  $k$. Now suppose $\phi : \Gamma_1 \to \Gamma_2$ is a partial
  automorphism of $\Gamma$. Both $\Gamma_1$ and $\Gamma_2$ are
  lattices in $N$, so $\gamma^k(\Gamma_j) = \Gamma_j \cap \gamma^k(N)$
  for $j=1,2$. It follows that
  \[
  \gamma^k(\Gamma_j) = \Gamma_j \cap \gamma^k(\Gamma) \text{ for }
  j=1,2.
  \]
  Clearly, $\phi( \gamma^k(\Gamma_1) ) = \gamma^k(\Gamma_2)$ for all
  $k$, from which it follows that $\gamma^k(\Gamma)$ is strongly
  commensuristic for all $k$.

  Consider the lower central series $\gamma_k(\Gamma)$. Then
  $\gamma_k(\Gamma)$ is Zariski-dense in $\gamma_k(N)$ for all $k$ by
  \cite[2.4]{borelbook}. Now suppose $\phi: \Gamma_1 \to \Gamma_2$ is
  a partial automorphism of $\Gamma$. Then $\gamma_k(\Gamma_j)$ is a
  lattice in $\gamma_k(N)$ for all $k$ for $j=1,2$. Since
  $\gamma_k(\Gamma_j) \leq \gamma_k( \Gamma )$, it follows that
  $\gamma_k(\Gamma_j) \leq \gamma_k( \Gamma)$ is of finite index for all
  $k$ for $j=1,2$. Since $\phi$ clearly takes $\gamma_k(\Gamma_1)$ to
  $\gamma_k(\Gamma_2)$ for all $k$, it follows that $[\phi] \cdot
  [\gamma_k(\Gamma)] = [\gamma_k(\Gamma)]$ for all $k$.
\end{proof}

\subsection{Commensurations are rational}
Let $N$ be a simply-connected nilpotent Lie group containing a lattice
$\Gamma$. Let $\nn$ denote the Lie algebra of $N$. Then $\nn$ admits
rational structure constants by \cite[2.12]{raghunathan}. It follows
that $\nn$ admits a basis, unique up to $\Q$-defined isomorphism, so
that $\log(\Gamma) \subseteq \nn(\Q)$. Further, $\Aut(\nn)$ is
identified with $\A(\R)$ for a $\Q$-defined algebraic subgroup $\A
\leq \GL(\nn \otimes \C)$ unique up to $\Q$-defined isomorphism. It is
a standard fact of Lie theory that the exponential map identifies
$\Aut(N)$ with $\Aut(\nn)$. This identifies $\Aut(N)$ with the real
points of a $\Q$-defined algebraic group $\A$. By abuse of notation,
we write $\Aut(N)(\Q)$ for the subgroup of $\Aut(N)$ corresponding to
$\A(\Q)$. The group $\Aut(N)(\Q)$ depends only on $N$ and $\Gamma$.

\begin{theorem} \label{nilpotentcomm} Let $\Gamma \leq N$ be a lattice
  in a simply-connected nilpotent Lie group. Identify $\Aut(N)$ with
  the real points of a $\Q$-defined algebraic group as above. Then
  there is an isomorphism
  \[
  \xi : \Comm(\Gamma) \to \Aut(N)(\Q).
  \]
\end{theorem}
To prove this theorem we will use the fact, due to Malcev, that
lattices in nilpotent groups are strongly rigid:
\begin{theorem}[{\cite[2.11]{raghunathan}}] \label{nilrigidity}
  Let $N_1$ and $N_2$ be two simply-connected nilpotent Lie groups,
  with lattices $\Gamma_1 \leq N_1$ and $\Gamma_2 \leq N_2$. Then
  every isomorphism $\Gamma_1 \to \Gamma_2$ extends uniquely to an
  isomorphism $N_1 \to N_2$.
\end{theorem}
\begin{proof}[Proof of Theorem \ref{nilpotentcomm}:]
  Suppose $\phi : \Gamma_1 \to \Gamma_2$ is a partial automorphism of
  $\Gamma$. Then $\phi$ extends to an automorphism $\Phi \in \Aut(N)$
  by Theorem \ref{nilrigidity}. Since $\log(\Gamma)$ is contained in
  $\nn(\Q)$, this extension is in $\Aut(N)(\Q)$. This gives an
  injective homomorphism
  \[
  \xi : \Comm(\Gamma) \to \Aut(N)(\Q).
  \]
  
  Now suppose $\Phi \in \Aut(N)(\Q)$. It is well-known (for example,
  see \cite[Chapter 2]{raghunathan}) that there is a $\Q$-defined
  unipotent algebraic group $\U$ so that $N \cong \U(\R)$ and
  $\Gamma$ is commensurable with $\U(\Z)$. Then $\Phi$ extends to a
  $\Q$-defined automorphism of $\U$. By \cite[10.14]{raghunathan} the
  group $\Phi(\Gamma)$ is commensurable with $\Gamma$, hence $\Phi$
  induces a commensuration of $\Gamma$. It follows that $\xi$ is
  surjective.
\end{proof}

\section{The algebraic hull of a polycyclic group} \label{hullsection}

\subsection{Polycyclic groups}

We briefly review the general theory of lattices in solvable Lie
groups. See \cite{raghunathan} for a general reference, and
\cite{segalbook} for the theory of polycyclic groups.
\begin{definition}
  A group $\Gamma$ is {\em polycyclic} if there is a subnormal series
  \begin{equation} \label{polycyclicseries} 1 \lhd \Gamma_1 \lhd
    \Gamma_2 \lhd \dotsb \lhd \Gamma \end{equation}
  so that $\Gamma_i / \Gamma_{i-1}$ is cyclic for each $i$.
\end{definition}
The {\em Hirsch number} of $\Gamma$, denoted $\rank(\Gamma)$, is the
number of $i$ such that $\Gamma_i / \Gamma_{i-1}$ is infinite
cyclic. Hirsch number is independent of choice of such subnormal
series, and is invariant under passage to finite index
subgroups. Every polycyclic group contains a finite index subgroup
admitting a subnormal series (\ref{polycyclicseries}) such that each
$\Gamma_i / \Gamma_{i-1}$ is infinite cyclic. Such a group is called
{\em strongly polycyclic}. It is well-known that every lattice in a
connected, simply-connected solvable Lie group is strongly polycyclic.

Every polycyclic group $\Gamma$ admits a unique maximal normal
nilpotent subgroup, called the {\em Fitting subgroup}, denoted
$\Fitt(\Gamma)$. If $\Gamma$ is a strongly polycyclic group, then
$\Fitt(\Gamma)$ is isomorphic to a lattice in a simply-connected
nilpotent Lie group $N$. By Theorem \ref{nilrigidity} conjugation
extends to a representation $\tilde \sigma : \Gamma \to \Aut(N)$. If
$\mathfrak{n}$ is the Lie algebra of $N$, then by identifying
$\Aut(N)$ with $\Aut(\nn) \subseteq \GL(\mathfrak{n})$ we have a
representation $\sigma : \Gamma \to \GL(\nn)$.
\begin{proposition}[{\cite[4.10]{raghunathan}}] \label{fittingunipotent}
  Let $\Gamma$ be strongly polycyclic, and $\sigma : \Gamma \to
  \GL(\nn)$ as above. Then
  \[
  \Fitt(\Gamma) = \left \{ \gamma \in \Gamma \suchthat
    \sigma(\gamma) \text{ is unipotent} \right \}. 
  \]
\end{proposition}
\begin{lemma} \label{fittingfi} Let $\Gamma$ be a strongly polycyclic
  group with $\Gamma_1 \leq \Gamma$ a subgroup of finite index. Then
  $\Fitt(\Gamma_1) = \Fitt(\Gamma) \cap \Gamma_1$.
\end{lemma}
\begin{proof}
  It is clear that $\Fitt(\Gamma_1) \cap \Gamma \leq \Fitt(\Gamma)$,
  so we have only to show that $\Fitt(\Gamma_1) \leq \Fitt(\Gamma)$.
  Let $N$ be the Lie group containing $\Fitt(\Gamma)$ as a lattice,
  and let $N_1$ be the Lie group containing $\Fitt(\Gamma_1)$ as a
  lattice. Then $\Gamma_1 \cap \Fitt(\Gamma)$ is a lattice in $N$. It
  follows that the inclusion $\Gamma_1 \cap \Fitt(\Gamma) \to
  \Fitt(\Gamma_1)$ extends to an embedding of Lie groups $\incl : N
  \to N_1$ by \cite[2.11]{raghunathan}. This gives an embedding of
  $\mathfrak{n}$ as a Lie subalgebra of $\mathfrak{n_1}$. Let $\sigma
  : \Gamma \to \GL(\nn)$ and $\sigma_1 : \Gamma_1 \to \GL(\nn_1)$ be
  as above. Then $\mathfrak{n}$ is invariant under $\sigma_1( \Gamma_1
  )$ because $\Fitt(\Gamma)$ is normal in $\Gamma$. Suppose $\gamma
  \in \Fitt(\Gamma_1)$. Then by Proposition \ref{fittingunipotent},
  $\sigma_1(\gamma)$ is unipotent. It follows that $\sigma(\gamma)$ is
  unipotent, and so $\gamma \in \Fitt(\Gamma)$ by Proposition
  \ref{fittingunipotent}.
\end{proof}
\begin{corollary} \label{fittingcomm}
  If $\Gamma$ is strongly polycyclic then $\Fitt(\Gamma)$ is strongly
  commensuristic in $\Gamma$.  \qed
\end{corollary}

\subsection{Algebraic hulls}

Our main tool for understanding commensurations of a polycyclic group
$\Gamma$ will be its algebraic hull. Roughly speaking, the algebraic
hull is algebraic group in which $\Gamma$ embeds Zariski-densely that
has minimal torus while having maximal unipotent radical. The
extremality conditions are important for the extension of
commensurations to algebraic automorphisms. The original construction
is due to Mostow \cite{mostowrep}, with an alternate construction in
\cite{raghunathan}.  More recently, algebraic hulls have been
constructed for certain virtually polycyclic groups by Baues in
\cite{bauesinfrasolv}. We will need only the classical algebraic hull.

\begin{definition}[Algebraic hull] \label{algebraichull} Suppose
  $\Gamma$ is a strongly polycyclic group. An {\em algebraic hull} of
  $\Gamma$ is a $\Q$-defined algebraic group $\bfh$ with an embedding
  $\incl : \Gamma \to \bfh(\Q)$ so that
  \begin{enumerate}[label=(H\arabic{enumi})]
  \item \label{density} \label{h1} $\incl(\Gamma)$ is Zariski-dense in
    $\bfh$,
  \item \label{strongradical} $Z_{\bfh}(\U_\bfh) \leq \U_\bfh$, where
    $\U_\bfh$ is the unipotent radical of $\bfh$,
  \item \label{dimension} $\dim(\U_\bfh) = \rank(\Gamma)$, and
  \item \label{commensurability} \label{h4} $\incl(\Gamma) \cap
    \bfh(\Z)$ is of finite index in $\incl(\Gamma)$.
  \end{enumerate}
\end{definition}
Algebraic hulls exist for all strongly polycyclic groups; see
\cite{raghunathan} for a construction.  The importance of the
algebraic hull is its uniqueness:
\begin{lemma}[{\cite[4.41]{raghunathan}}] \label{hullextension}
  Suppose $\Gamma_1$ and $\Gamma_2$ are two strongly polycyclic
  groups, and $\phi: \Gamma_1 \to \Gamma_2$ is an isomorphism. Let
  $\incl_1 : \Gamma_1 \to \bfh_1$ and $\incl_2 : \Gamma_2 \to \bfh_2$ be
  algebraic hulls for $\Gamma_1$ and $\Gamma_2$, respectively. Then
  $\phi$ extends to a $\Q$-defined isomorphism $\Phi: \bfh_1 \to
  \bfh_2$.
\end{lemma}

We wish to use rigidity of the algebraic hull to construct an
embedding of $\Comm(\Gamma)$ into $\Aut_\Q(\bfh)$ analogous to the use
of Malcev rigidity in Theorem \ref{nilpotentcomm}. For this, the
natural setting is the Zariski-connected component of the identity of
the algebraic hull.
\begin{definition}[Virtual algebraic hull] \label{valghull} Let
  $\Gamma$ be a virtually polycyclic group. A {\em virtual algebraic
    hull} of $\Gamma$ is a triple $(\bfh, \Delta, i)$, where $\bfh$ is
  a $\Q$-defined algebraic group, $\Delta$ is a finite index subgroup
  of $\Gamma$, and $\incl:\Delta \to \bfh(\Q)$ is an injective
  homomorphism so that
  \begin{enumerate}
  \item $\bfh$ is connected, and
  \item $\bfh$ with the embedding $\incl$ is an algebraic hull of $\Delta$.
  \end{enumerate}
\end{definition}
\begin{lemma} \label{vah} Every virtually polycyclic
  group has a virtual algebraic hull.
\end{lemma}
\begin{proof}
  Suppose $\Gamma$ is virtually polycyclic. Let $\widetilde\Gamma \leq
  \Gamma$ be any finite index strongly polycyclic subgroup. Let
  $\widetilde\bfh$ be an algebraic hull for $\widetilde\Gamma$. Then
  the identity component $\widetilde\bfh^0$ is of finite index in
  $\widetilde\bfh$. Let $\Delta = \widetilde\Gamma \cap
  \widetilde\bfh^0$ and $\bfh = \widetilde\bfh^0$. It is easy to
  verify that $\bfh$ with the given inclusion of $\Delta$ in
  $\bfh(\Q)$ is a virtual algebraic hull of $\Gamma$.
\end{proof}

We will often abuse notation and refer to the algebraic group $\bfh$
of Definition \ref{valghull} as the virtual algebraic hull of
$\Gamma$, leaving the subgroup $\Delta$ and the inclusion $\incl: \Delta
\to \bfh(\Q)$ implicit.

\begin{lemma}
  Let $\Gamma$ be virtually polycyclic with virtual algebraic hull
  $(\bfh, \Delta, \incl)$. The algebraic group $\bfh$ is unique up to
  $\Q$-defined isomorphism.
\end{lemma}
\begin{proof}
  Suppose $(\bfh_1, \Gamma_1, \incl_1)$ and $(\bfh_2, \Gamma_2,
  \incl_2)$ are two virtual algebraic hulls of $\Gamma$. 
  Then $\Gamma_1 \cap \Gamma_2$ is of finite index
  in both $\Gamma_1$ and $\Gamma_2$. Because $\bfh_1$ and $\bfh_2$ are
  connected, both $\restr{\incl_1}{\Gamma_1 \cap \Gamma_2}$ and
  $\restr{\incl_2}{\Gamma_1 \cap \Gamma_2}$ satisfy \ref{h1}--\ref{h4}
  for $\Gamma_1 \cap \Gamma_2$ in place of $\Gamma$. It follows from
  Lemma \ref{hullextension} that there is a $\Q$-defined isomorphism
  $\Phi : \bfh_1 \to \bfh_2$ extending $\incl_2 \circ
  \restr{\incl_1}{\Gamma_1 \cap \Gamma_2}^{-1}$.
\end{proof}

\begin{definition}[Fitting subgroup] \label{hullfitting} Suppose
  $(\bfh, \Delta, \incl)$ is a virtual algebraic hull of a virtually
  polycyclic group $\Gamma$. Define $\Fitt(\bfh)$, the {\em Fitting
    subgroup} of $\bfh$, to be the Zariski-closure of $\Fitt(\Delta)$
  in $\bfh$. Note that $\Fitt(\bfh)$ depends the inclusion
  $\incl : \Delta \to \bfh(\Q)$; we suppress this dependence from our
  notation as the embedding $\incl$ is implicit in the choice of
  virtual algebraic hull $\bfh$.
\end{definition}

Note that $[\bfh, \bfh] \leq \Fitt(\bfh)$, by
\cite[4.6]{bauesgrunewald}.

\begin{lemma} \label{solvcommembedding} Suppose $\Gamma$ is virtually
  polycyclic with virtual algebraic hull $(\bfh, \Delta, i)$. There is
  an embedding
  \begin{equation} \label{commembedding} \xi : \Comm(\Gamma) \to
    \Aut_\Q(\bfh). \end{equation}
\end{lemma}
\begin{proof}
  Suppose $\phi : \Delta_1 \to \Delta_2$ is a
  partial automorphism of $\Delta$. Then $\bfh$ is an algebraic hull
  for both $\Delta_1$ and $\Delta_2$ by connectedness, so $\phi$
  extends to $\Phi \in \Aut_\Q(\bfh)$. Equivalent partial automorphisms
  clearly give rise to equal extensions. The assignment $\phi \mapsto
  \Phi$ gives an injective homomorphism $\Comm(\Delta) \to
  \Aut_\Q(\bfh)$ by density of $\Delta_1$ and $\Delta_2$. The proof is
  complete since $\Comm(\Gamma) \cong \Comm(\Delta)$.
\end{proof}

There is an analogous construction of algebraic hulls for
simply-connected solvable Lie groups $G$, though they are only
$\R$-defined rather than $\Q$-defined.

\begin{definition}[Algebraic hull] \label{realalgebraichull} Suppose
  $G$ is a simply-connected solvable Lie group. A {\em real algebraic
    hull} of $G$ is an $\R$-defined algebraic group $\bfh$ with an
  embedding $\incl : G \to \bfh(\R)$ so that
  \begin{enumerate}
  \item \label{realdensity} $\incl(G)$ is Zariski-dense in
    $\bfh$,
  \item \label{realstrongradical} $Z_{\bfh}(\U_\bfh) \leq \U_\bfh$, where
    $\U_\bfh$ is the unipotent radical of $\bfh$, and
  \item \label{realdimension} $\dim(\U_\bfh) = \dim(G)$.
  \end{enumerate}
\end{definition}

The real algebraic hull of the group $G$ may be strictly larger than
the algebraic hull of a lattice $\Gamma \leq G$. See
\cite{bauesklopsch} for a detailed discussion of the relationship
between the algebraic hull of a lattice and the real algebraic hull of
the ambient Lie group. We use this theory in \S\ref{gensection}.

\subsection{Unipotent shadow}

Much of the theory of lattices in solvable Lie groups builds on the
much easier theory of lattices in nilpotent Lie group. A common tool
is the {\em unipotent shadow}. 
The following proposition summarizes the theory of unipotent shadows
of strongly polycyclic groups in algebraic hulls, as explained in
Sections 5 and 7 of \cite{bauesgrunewald}. For the reader's
convenience we include a sketch of a proof.

\begin{proposition}[\cite{bauesgrunewald}] \label{unipotentshadow} Suppose
  $\Gamma$ is a virtually polycyclic group with virtual algebraic hull
  $\bfh$. Let $\F$ be the Fitting subgroup of $\bfh$. There is a
  strongly polycyclic subgroup $\Lambda \leq \bfh(\Q)$ abstractly
  commensurable with $\Gamma$ so that:
  \begin{enumerate}
  \item There is a nilpotent subgroup $C\leq \Lambda$ so that $\Lambda
    = \Fitt(\Lambda) \cdot C$.
  \item There is a $\Q$-defined maximal torus $\bft \leq \bfh$ with
    centralizer $\D \leq \bfh$ so that $C = \Lambda \cap \D$ and $C$
    is Zariski-dense in $\D$.
  \item The subgroup $\theta \leq \U_\bfh(\Q)$ generated by
    $\Fitt(\Lambda)$ and $C_u$, the group of unipotent parts of
    elements of $C$, is a finitely generated subgroup Zariski-dense in
    $\U_\bfh$, such that $\Fitt(\Lambda) = \theta \cap \F$.
  \end{enumerate}
\end{proposition}

\begin{proof}[Sketch of proof:] 
  Let $\Delta$ be a strongly polycyclic subgroup of $\Gamma$ so that
  $\bfh$ is an algebraic hull of $\Delta$. Fix any maximal
  $\Q$-defined torus $\bft \leq \bfh$, and let $\D$ be the normalizer
  of $\bft$ in $\bfh$. Then $\D$ is a connected nilpotent $\Q$-defined
  subgroup of $\bfh$ that centralizes $\bft$.  By replacing
  $\Fitt(\Delta)$ with a finite index supergroup, we obtain a strongly
  polycyclic group $\Lambda \leq \bfh(\Q)$ commensurable with $\Delta$
  for which the group $C = \Lambda \cap \D$ is Zariski-dense in $\D$
  and satisfies $\Lambda = \Fitt(\Lambda) \cdot C$. The group
  $\Lambda$ is called a {\em thickening} of $\Delta$, and $C$ is
  called a {\em nilpotent supplement} in $\Lambda$.

  We now want to construct the group $\theta$ by taking the unipotent
  parts of elements of $\Lambda$. For every $c\in \D$, let $c_s$ and
  $c_u$ denote the semisimple and unipotent parts, respectively, of
  its Jordan decomposition in $\D$. Because $\D$ centralizes $\bft$,
  the map $c \mapsto c_u$ is a homomorphism $\D \to \U_\bfh$. Define
  $\theta$ to be the subgroup of $\U_\bfh(\Q)$ generated by
  $\Fitt(\Lambda)$ and $C_u$. By replacing $\Lambda$ with a further
  thickening, we can guarantee that $\theta \cap \F =
  \Fitt(\Lambda)$. Such a group $\theta$ is called a {\em good
    unipotent shadow}.
\end{proof}

\subsection{Algebraic structure of $\Aut(\bfh)$} \label{autstructure}
Suppose $\Gamma \leq G$ is a lattice in a simply-connected solvable
Lie group, and let $\bfh$ be its virtual algebraic hull. We recall the
structure of $\Aut_\Q(\bfh)$ explained in Section 3 of
\cite{bauesgrunewald}. Let $\U$ be the unipotent radical of $\bfh$. Fix
a $\Q$-defined maximal torus $\bft \leq \bfh$. There is a
decomposition $\bfh = \U \rtimes \bft$. Define
\begin{equation} \Aut(\bfh)_\bft = \{ \Phi \in \Aut(\bfh) \mid \Phi(\bft) =
\bft \} . \end{equation}
By property \ref{strongradical} of the algebraic hull, the restriction
map $\Aut(\bfh)_\bft \to \Aut(\U)$ is injective. Its image is a
$\Q$-defined closed subgroup of $\Aut(\U)$. The map 
\begin{equation} \label{authmap}
  \begin{split} \Theta : \U \rtimes \Aut(\bfh)_\bft
    &\to \Aut(\bfh) \\ (u, \Phi) &\mapsto \Inn_u \circ \Phi 
  \end{split} 
\end{equation} is a surjection with $\Q$-defined
kernel. The quotient $\U \rtimes \Aut(\bfh)_\bft / \ker( \Theta )$ is
a $\Q$-defined algebraic group, which gives $\Aut(\bfh)$ the structure
of a $\Q$-defined algebraic group. Because $\ker(\Theta)$ is
unipotent, it follows from the discussion of
\cite[2.2.3]{platonovrapinchuk} (see also \cite[3.6]{bauesgrunewald})
that there is a group isomorphism
\begin{equation} \label{qdefinedautos}
  \Aut_\Q(\bfh) \cong \U(\Q) \rtimes \Aut(\bfh)_\bft (\Q) / (\ker
\Theta) (\Q).
\end{equation}
Thus the algebraic structure of $\Aut(\bfh)$ is such that
$\Aut_\Q(\bfh) = \Aut(\bfh)(\Q)$.

\subsection{A finite index subgroup of $\Comm(\Gamma)$}
\label{commhfsection}
Let $\Gamma$, $\bfh$, and $\U$ be as above. Let $\F =
\Fitt(\bfh)$. Define
\begin{equation} \ahu = \left \{ \Phi \in \Aut(\bfh) \suchthat
  \restr{\Phi}{\bfh / \U} = \Id_{\bfh/\U}  \right \} . \end{equation}
\begin{lemma} \label{ahufinite}
  The subgroup $\ahu \leq \Aut(\bfh)$ is of finite index.
\end{lemma}
\begin{proof}
  The quotient $\bfh / \U$ is a $\Q$-defined torus. By Lemma
  \ref{toririgidity}, the identity component $\Aut(\bfh)^0$ acts
  trivially on the torus $\bfh/\U$, and so $\Aut(\bfh)^0 \leq
  \ahu$. The claim follows since $[ \Aut(\bfh)^0 : \Aut(\bfh) ] <
  \infty$.
\end{proof}

Let $N_{\Aut(\bfh)}(\F)$ denote the subgroup of $\Aut(\bfh)$ preserving
$\F$. Define
\begin{equation} \ahf = \left \{ \Phi \in N_{\Aut(\bfh)}(\F) \suchthat
    \restr{\Phi}{\bfh / \F} = \Id_{\bfh/\F} \right \} . \end{equation}
By Corollary \ref{fittingcomm}, the image of the map $\xi:
\Comm(\Gamma) \to \Aut(\bfh)$ preserves $\F$. Define
\begin{equation} \Comm_{\bfh \mid \F}(\Gamma) =
  \xi^{-1}(\ahf). \end{equation}
\begin{lemma} \label{commhffinite}
  $[ \Comm(\Gamma) : \Comm_{\bfh \mid \F}(\Gamma) ] < \infty$.
\end{lemma}
\begin{proof}
  By Lemma \ref{ahufinite}, it suffices to show that $\Comm_{\bfh \mid
    \F}(\Gamma) = \Comm(\Gamma) \cap \ahu$. Since $\ahf \leq \ahu$, it
  is clear that $\Comm_{\bfh \mid \F}(\Gamma) \leq \Comm(\Gamma) \cap
  \ahu$. On the other hand, suppose that $[\phi] \in \Comm(\Gamma)
  \cap \ahu$. Without loss of generality, assume that $\phi$ is a
  partial automorphism of a finite index subgroup $\Delta \leq \Gamma$
  for which $\bfh$ is an algebraic hull. By Proposition
  \ref{fittingunipotent}, we have that $\Delta \cap \U =
  \Fitt(\Delta)$. It follows that if $\phi(\gamma) \gamma^{-1} \in \U$
  for some $\gamma \in \Delta$, then $\phi(\gamma) \gamma^{-1} \in
  \Fitt(\Delta)$. Therefore $[\phi] \in \Comm_{\bfh \mid \F}(\Gamma)$.
\end{proof}

The structure of $\ahf$ can be made more explicit, following Section
3.3 of \cite{bauesgrunewald}. Let $\bft$ denote a maximal $\Q$-defined
torus in $\bfh$. Define
\begin{align}
  \as &= \left\{ \Phi \in \ahf \mid \Phi(\bft) = \bft, \;
    \restr{\Phi}{\bft} = \operatorname{id}_\bft \right\}, \\
  \Inn_\F^\bfh &= \left \{ \Phi\in \Aut(\bfh) \mid \Phi(x) = fxf^{-1}
    \text{ for some } f\in \F   \right \}.
\end{align}
Clearly $\Inn_\F^\bfh$ and $\as$ are both $\Q$-defined subgroups of
$\ahf$, and $\Inn_\F^\bfh$ is normal in $\ahf$. Let $(\ahf)_\Q$ denote
the group of $\Q$-defined automorphisms in $\ahf$. Because any two
maximal $\Q$-defined tori are conjugate by an element of $[\bfh,
\bfh](\Q) \leq \F(\Q)$, we have
\begin{lemma} \label{ahflemma}
  $\ahf = \Inn_\F^\bfh \cdot \as$. Moreover, $(\ahf)_\Q =
  \Inn_{\F}^\bfh(\Q) \cdot \as(\Q)$.
\end{lemma}
\begin{proof}
  See \cite[3.13]{bauesgrunewald}. The latter statement follows from
  equation (\ref{qdefinedautos}); cf. \cite[3.6]{bauesgrunewald}.
\end{proof}

\section{Commensurations of lattices in solvable
  groups} \label{solvsection}

\subsection{Example: \textsc{Sol} lattice} Let $\psi : \Z^2 \to
\Z^2$ be the automorphism defined by $\psi(1,0) = (2,1)$ and
$\psi(0,1) = (1,1)$. Let $C$ be the infinite cyclic group generated by
$\psi$, and define $\Gamma = \Z^2 \rtimes C$. Note that $\Gamma$ is a
lattice in 3-dimensional \textsc{Sol} geometry. We have that
$\Fitt(\Gamma) = \Z^2$, so there are induced maps 
\[
r : \Comm(\Gamma) \to \Comm(\Z^2) \cong \GL_2(\Q)
\]
and 
\[
\pi : \Comm(\Gamma) \to \Comm( C ) \cong \Q^*.
\]
Suppose $\phi: \Gamma_1\to \Gamma_2$ is a partial automorphism of
$\Gamma$. There are nonzero $p, q$ so that $\pi(\phi)[\psi^q] =
[\psi^p]$. Using the fact that $\phi$ is an isomorphism, we have
\begin{equation} \label{solexampleeqn} \phi( \psi^q(v) ) = \psi^p(
  \phi(v) ) \end{equation} for all $v\in \Gamma_1 \cap \Z^2$.  Since
$\Gamma_1 \cap \Z^2$ spans $\Z^2 \otimes \Q$, it follows that
$r(\phi)$ conjugates $\psi^q$ to $\psi^p$ in $\GL_2(\Q)$. Therefore $p
= \pm q$ since $\psi$ has an eigenvalue not on the unit circle. It
follows that there is an index 2 subgroup $\Comm^+(\Gamma)$ so that
$\pi$ is trivial when restricted to $\Comm^+(\Gamma)$.

Let $Z_{\GL_2(\Q)}(\psi)$ denote the centralizer of $\psi$ in
$\GL_2(\Q)$. From (\ref{solexampleeqn}) we see that $r(\phi) \in
Z_{\GL_2(\Q)}(\psi)$ for all $\phi \in \Comm^+(\Gamma)$. Moreover, it
is clear that the induced map $\bar r : \Comm^+(\Gamma) \to
Z_{\GL_2(\Q)}(\psi)$ is surjective. Let $K = \ker(\bar r)$. Every
$\phi \in K$ is of the form $\phi( v, \psi^p ) = (v + \rho(\psi^p) ,
\psi^p )$ for a cocycle $\rho : H \to \Z^2$ defined on some finite
index subgroup $H \leq C$. One can show that
\[
K = \varprojlim_{[C : H] < \infty} H^1( H , \Z^2 ) \cong H^1( C,
\Q^2) \cong \Q^2.
\]

Therefore $\Comm^+(\Gamma)$ satisfies the short exact sequence
\[
1 \to \Q^2 \to \Comm^+(\Gamma) \to Z_{\GL_2(\Q)}(\psi) \to 1.
\]
This sequence splits, and the action of $Z_{\GL_2(\Q)}(\psi)$ on
$\Q^2$ is the standard action.

\subsection{Commensurations of solvable lattices are rational}

We continue to use the notation developed in
\S\ref{hullsection}. Given a lattice $\Gamma$ in a connected,
simply-connected solvable Lie group, let $\bfh$ denote its virtual
algebraic hull with Fitting subgroup $\F$ and $\Q$-defined maximal
torus $\bft$. Then $\ahf$ denotes the group of automorphisms of $\bfh$
preserving $\F$ and trivial on $\bfh/\F$. Let $\Inn_\F^\bfh$ denote
the group of automorphisms of $\bfh$ induced by conjugation by
elements of $\F$, and $\as$ denote the group of automorphisms fixing
$\bft$.

\begin{theorem} \label{commhfdescription} Let $\Gamma$ be a virtually
  polycyclic group. Let $\bfh$ be the virtual algebraic hull of
  $\Gamma$, with $\F = \Fitt(\bfh)$. The map $\xi : \Comm(\Gamma) \to
  \Aut(\bfh)$ induces an isomorphism of groups
  \[
  \Comm_{\bfh\mid \F}(\Gamma) \cong (\ahf)_\Q.
  \]
\end{theorem}

The proof of the theorem is in two steps. First we show that
$\Inn_\F^\bfh(\Q) \leq \xi(\Comm(\Gamma))$, and second that $\as(\Q)
\leq \xi(\Comm(\Gamma))$. The unipotent shadow will be our main
tool. Before we start the proof of Theorem \ref{commhfdescription}, we
note the following technical lemma, which will be used again in
\S\ref{gensection}.
\begin{lemma} \label{shadowaction} Let $\U$ be a $\Q$-defined
  unipotent algebraic group and $\theta' \leq \U(\Q)$ be a finitely
  generated, Zariski-dense subgroup. Let $P$ be a group acting on $\U$
  by algebraic group automorphisms preserving $\theta'$. Suppose $f\in
  \U(\Q)$. There is some finite index subgroup $P'' \leq P$ so that $f
  (p \cdot f^{-1}) \in \theta'$ for all $p\in P''$.
\end{lemma}
\begin{proof}
  Let $\Lambda$ be the group generated by $\theta'$ and $f$. Then
  $\Lambda$ is commensurable with $\U(\Z)$, hence contains $\theta'$
  as a subgroup of finite index $d$. It is not hard to see that there
  are finitely many subgroups of $\U(\Q)$ containing $\theta'$ with
  index $d$ (see \cite[Ch6]{segalbook} and
  \cite[6.3]{bauesgrunewald}). The group $P$ permutes these subgroups,
  hence there is a subgroup $P' \leq P$ preserving $\Lambda$. Because
  $\theta'$ has finite index in $\Lambda$ and is preserved by $P$, there
  is a further finite index subgroup $P''\leq P$ that acts trivially
  on the coset space $\Lambda / \theta'$. This completes the proof.
\end{proof}

\begin{proof}[Proof of Theorem \ref{commhfdescription}]

  Given $\Gamma$ and $\bfh$ as in the theorem, let $\U =
  \U_\bfh$. Find a strongly polycyclic subgroup $\Lambda \leq
  \bfh(\Q)$ abstractly commensurable with $\Gamma$, along with $\bft$,
  $\D$, $C$, and $\theta$ as in Proposition
  \ref{unipotentshadow}. That is, $\bft$ is a maximal $\Q$-defined
  torus, $\D$ is the centralizer of $\bft$ containing $C =
  \Lambda \cap \D$ as a Zariski-dense subgroup, and $\theta \leq
  \U(\Q)$ is a good unipotent shadow of $\Lambda$.

  By Lemma \ref{solvcommembedding} there is an embedding
  \[
  \xi : \Comm(\Gamma) \to \Aut_\Q(\bfh).
  \]
  By definition of $\Comm_{\bfh\mid \F}(\Gamma)$, this restricts to an
  embedding
  \[
  \hat \xi : \Comm_{\bfh \mid \F} (\Gamma) \to (\ahf)_\Q.
  \]
  There is a decomposition $(\ahf)_\Q = \Inn_\F^\bfh (\Q) \cdot
  \as(\Q)$ by Lemma \ref{ahflemma}. We have only to show that both
  $\Inn_\F^\bfh(\Q)$ and $\as(\Q)$ are in the image of $\hat \xi$.

  \bold{Claim 1:} $\Inn_\F^\bfh(\Q) \leq \xi(\Comm(\Gamma))$.
  
  \italics{Proof of Claim 1:}
    Suppose $\Phi \in \Inn_\F^\bfh(\Q)$. Then there is some $f\in
    \F(\Q)$ so that $\Phi(x) = fxf^{-1}$ for all $x\in \bfh$. Because
    $\theta$ is Zariski-dense in $\U$, conjugation by $f$ induces a
    commensuration of $\theta$ by Theorem \ref{nilpotentcomm}. Let
    $\theta_1$ and $\theta_2$ be finite index subgroups of $\theta$ so
    that $\Phi(\theta_1) = \theta_2$. Let $C' \leq C$ be a finite
    index subgroup normalizing both $\theta_1$ and $\theta_2$. By
    Lemma \ref{shadowaction}, applied with $\theta' = \theta_1 \cap
    \theta_2$ and $P = C'$, there is some finite index subgroup $C''
    \leq C'$ so that
    \begin{equation} \label{inneqn} f c f^{-1} c^{-1} \in \theta_1 \cap
      \theta_2 \end{equation} for all $c \in C''$. Because $\F$ is
    normal in $\U$, for all $c\in C''$ we have $f c f^{-1} c^{-1} \in
    \F$. By (\ref{inneqn}) and the fact that $\theta \cap \F =
    \Fitt(\Lambda)$, for all $c\in C''$ we have
    \begin{equation} \label{inneqn2} f c f^{-1} c^{-1} \in
      \Fitt(\Lambda) \cap \theta_1 \cap \theta_2. \end{equation}
    
    Let $F_1 = \theta_1 \cap \Fitt(\Lambda)$ and $F_2 = \theta_2 \cap
    \Fitt(\Lambda)$. Then $\Phi$ induces an isomorphism $F_1 \to
    F_2$. Because $C''$ normalizes both $F_1$ and $F_2$, we may form
    subgroups $\Lambda_1 = F_1 C''$ and $\Lambda_2 = F_2 C''$, both of
    which are of finite index in $\Lambda$. We claim that $\Phi$
    induces an isomorphism $\Lambda_1 \to \Lambda_2$. Suppose $f_1 \in
    F_1$ and $c_1 \in C''$. Then $f f_1 f^{-1} \in F_2$ by definition
    of $\theta_1$ and $\theta_2$, and $f c_1 f^{-1} = f_2 c_1$ for
    some $f_2 \in F_2$ by (\ref{inneqn2}). Therefore
    \[
    f f_1 c_1 f^{-1} = f f_1 f^{-1} f c_1 f^{-1} \in F_2 C''.
    \]
    It follows that $\Phi$ induces an injection $\Lambda_1 \to
    \Lambda_2$. Note that (\ref{inneqn2}) holds for all $c\in C''$
    with $f$ replaced by $f^{-1}$. Similar reasoning then gives that
    $\Phi^{-1}$ induces an injection $\Lambda_2 \to \Lambda_1$. Thus
    $\Phi$ induces a partial automorphism $\Lambda_1 \to \Lambda_2$ of
    $\Lambda$, and so induces a commensuration of $\Gamma$. This
    completes the proof of Claim 1.

  \bold{Claim 2:} $\as(\Q) \leq \xi(\Comm(\Gamma))$.

  \italics{Proof of Claim 2:}
  Suppose $\Phi \in \as(\Q)$. Then $\Phi$ corresponds to a
  $\Q$-defined map under the restriction $\as \to \Aut(\U)$, so $\Phi$
  induces a partial automorphism $\theta_1 \to \theta_2$ of $\theta$
  by Theorem \ref{nilpotentcomm}. The map $C \to C_u$ is a
  homomorphism. Define a finite index subgroup
  \[
  C_1 = \left \{ c \in C \suchthat c_u \in \theta_1 \right \}
  \leq C. 
  \]
  Take any $c_1 \in C_1$, and write $c_1 = u_1 s$ for $u_1 \in
  \theta_1$ and $s\in \bft$. Since $\Phi \in \ahf$, there is some $f
  \in \F(\Q)$ so that $\Phi( u_1 ) = f u_1$. Since $\Phi \in \as$,
  we have
  \[
  \Phi( c_1 ) = \Phi(u_1) \Phi( s ) = f u_1 s = f c_1 . 
  \] 
  Both $\Phi(u_1)$ and $u_1$ are in $\theta$, so $f \in \theta \cap
  \F = \Fitt(\Lambda)$. Therefore $\Phi(c_1) \in \Lambda$. Since
  $\Phi$ preserves $\bft$, it also preserves $\D$. Therefore
  $\Phi(C_1) \leq C$ since $\Lambda \cap \D = C$.
  
  Define
  \[
  C_2 = \left \{ c \in C \suchthat c_u \in \theta_2 \right \}
  \leq C. 
  \]
  It is evident from the definitions of $\theta_1$ and $\theta_2$
  that $\Phi(C_1) \leq C_2$. Applying the same logic as above to
  $\Phi^{-1}$, we conclude that $\Phi(C_1) = C_2$. Therefore $\Phi$
  induces a partial automorphism $C_1 \to C_2$ of $C$.
  
  Since $\Phi$ preserves $\F$, it induces a partial automorphism
  $F_1 \to F_2$ of $\Fitt(\Lambda)$. Without loss of generality,
  suppose $F_1$ is characteristic in $\Fitt(\Lambda)$. Then $F_1
  C_1$ and $F_2 C_2$ are both finite index subgroups of
  $\Lambda$. So $\Phi$ induces a partial automorphism $F_1 C_2 \to
  F_2 C_2$ of $\Lambda$, and hence a commensuration of
  $\Gamma$. This completes the proof of Claim 2.

  \medskip Claims 1 and 2 show that $\hat \xi$ is surjective, and
  therefore $\hat \xi$ exhibits an isomorphism $\Comm_{\bfh\mid
    \F}(\Gamma) \cong (\ahf)_\Q$. This completes the proof of
  Theorem \ref{commhfdescription}.
\end{proof}

\begin{proof}[Proof of Theorem \ref{introsolv}:]
  Let $\bfh$ be the virtual algebraic hull of $\Gamma$. By Lemma
  \ref{solvcommembedding} there is an embedding
  \[
  \xi : \Comm(\Gamma) \to \Aut(\bfh)(\Q),
  \]
  where $\Aut(\bfh)$ has the structure of an algebraic group as
  described in Section \ref{autstructure}. Let $\mathcal{A}_\Gamma$ be
  the Zariski-closure of $\xi(\Comm(\Gamma))$ in $\Aut(\bfh)$. Then
  $\mathcal{A}_\Gamma$ is a $\Q$-defined algebraic group by
  Proposition \ref{closuredefinition}. Now take any $\Psi \in
  \mathcal{A}_\Gamma(\Q)$. Take any element $\Phi \in \xi( \Comm(
  \Gamma ) )$ so that $\Psi \circ \Phi^{-1} \in
  \mathcal{A}_\Gamma^0(\Q)$. We have $\mathcal{A}_\Gamma^0 \leq \ahu$
  by Lemma \ref{ahufinite} and then $\mathcal{A}_\Gamma^0 \leq \ahf$
  by Lemma \ref{commhffinite}. Therefore $\Psi \circ \Phi^{-1} \in
  \ahf(\Q)$. It follows from Theorem \ref{commhfdescription} that
  $\Psi \in \xi( \Comm( \Gamma ) )$, hence the isomorphism
  \[
  \Comm(\Gamma) \cong \mathcal{A}_\Gamma(\Q).
  \]

  We have only to show that the image of $\Aut(\Gamma)$ in
  $\Aut(\bfh)$ is commensurable with $\mathcal{A}_\Gamma(\Z)$.  Let $F
  = \Fitt(\Gamma)$ and define
  \[
  A_{\Gamma \mid F} = \left\{ \phi \in \Aut(\Gamma) \suchthat
    \restr{ \phi }{ \Gamma / F} = \restr{\Id}{ \Gamma / F }
  \right\}. 
  \]
  The proof of Lemma \ref{commhffinite} shows that $A_{\Gamma \mid F}$
  is finite index in $\Aut(\Gamma)$; see also
  \cite[9.1]{bauesgrunewald}.  The group $A_{\Gamma \mid F}$ is
  commensurable with $\mathcal{A}_{\bfh\mid \F}(\Z)$ by
  \cite[8.9]{bauesgrunewald}, so the result follows.
\end{proof}

We conclude this section with a result relating the structure of
$\mathcal{A}_\Gamma$ to that of $\Aut(G)$ for certain solvable groups
$G$. This strengthens the analogy with semisimple groups; compare with
Theorem \ref{semicomm} below.

\begin{definition} \label{unipconn} Let $\Nil(G)$ denote the maximal
  normal nilpotent subgroup of $G$. A solvable Lie group $G$ is {\em
    unipotently connected} if $\Nil(G)$ is connected.
\end{definition}

\begin{proposition} \label{thesisadded1} Suppose $G$ is a connected,
  simply-connected, unipotently connected solvable Lie group. Let
  $\Gamma\leq G$ be a Zariski-dense lattice and let
  $\mathcal{A}_\Gamma$ be the group such that $\mathcal{A}_\Gamma(\Q)
  \cong \Comm(\Gamma)$. Then
  \[
  \mathcal{A}_\Gamma(\R) \doteq \Aut(G).
  \]
\end{proposition}

\begin{proof}
  Let $\bfh$ be the real algebraic hull of $G$. By
  \cite[3.11]{bauesklopsch} the group $\bfh$ is also an algebraic hull
  for $\Gamma$. It further follows that $\F(\R) = \Nil(G)$ by
  \cite[5.4]{bauesklopsch}. For any $\Phi\in \ahf(\R)$, there is some
  $f\in \F(\R)$ such that $\Phi(g) = fg$ for all $g\in \bfh$.
  Therefore $\ahf(\R)$ preserves $G \leq \bfh(\R)$, and so $\ahf(\R)
  \leq \Aut(G)$. In fact, $[\Aut(G) : \ahf(\R)] < \infty$ by
  \cite[6.9]{bauesklopsch}. The result follows because $\ahf$ is a
  subgroup of finite index in $\mathcal{A}_\Gamma$.
\end{proof}

Every lattice in a connected, simply-connected solvable Lie group
virtually embeds as a Zariski-dense lattice in a connected,
simply-connected, unipotently connected solvable Lie group $G'$
(cf. \cite[5.3]{bauesklopsch}). Therefore we have:

\begin{corollary} \label{thesisadded2} Let $\Gamma$ be a lattice in a
  connected, simply-connected Lie group $G$. Let $\mathcal{A}_\Gamma$
  denote the algebraic group such that $\mathcal{A}_\Gamma(\Q) =
  \Comm(\Gamma)$. Then $\Gamma$ virtually embeds as a lattice in a Lie
  group $G'$ such that $\mathcal{A}_\Gamma(\R) \doteq \Aut(G')$.
\end{corollary}

\section{Commensurations of lattices in semisimple
  groups} \label{semisection} 

Abstract commensurators of lattices in semisimple Lie groups not
isogenous to $\PSL_2(\R)$ are fairly well understood, by work of
Borel, Mostow, and Margulis. For example, see the first section of
\cite{acampoburger}. We recall the basic results here for
completeness.

\subsection{Arithmetic lattices in semisimple groups}
\begin{definition} \label{arithmeticdefn} Suppose $\Gamma \leq S$ is a
  lattice in a semisimple Lie group with trivial center and no compact
  factors. We say that $\Gamma$ is {\em arithmetic} if there is a
  $\Q$-defined semisimple algebraic group $\bfs$ and a surjective
  homomorphism $f : \bfs(\R)^0 \to S$ with compact kernel such that $f(
  \bfs(\Z) \cap \bfs(\R)^0 )$ and $\Gamma$ are commensurable. 
\end{definition}
Note that $\bfs$ may be chosen to be simply-connected, and that
$\Gamma \doteq \bfs(\Z)$ by Proposition \ref{resfinite}. Hence, to
compute the abstract commensurators of arithmetic lattices in
semisimple Lie groups, it suffices to consider groups of the form
$\bfs(\Z)$ for a simply-connected $\Q$-defined semisimple algebraic
group $\bfs$.

Recall that a $\Q$-defined, connected, semisimple algebraic group
$\bfs$ is {\em without $\Q$-compact factors} if there is no
nontrivial, $\Q$-defined, connected, normal subgroup $\bfn \leq \bfs$
such that $\bfn(\R)$ is compact. Note that given any $\Q$-defined
connected, simply-connected, semisimple algebraic group, there is a
$\Q$-defined, connected, simply-connected, semisimple algebraic group
$\bfs'$ without $\Q$-compact factors such that $\bfs(\Z)$ and
$\bfs'(\Z)$ are abstractly commensurable.

If $\bfs$ is a $\Q$-defined semisimple algebraic group,
then $\Aut(\bfs)$, the group of automorphisms of $\bfs$ as an
algebraic group, has the structure of a $\Q$-defined algebraic
group such that $\Aut(\bfs)_\Q \cong \Aut(\bfs)(\Q)$; see
\cite[5.7.2]{tits}.

\begin{proposition} \label{semicommqpointspart} Suppose $\bfs$ is a
  $\Q$-defined, connected, simply-connected, semisimple algebraic group
  without $\Q$-compact factors. Then there is a canonical inclusion
  \[
  \Xi : \Aut( \bfs )(\Q) \into \Comm( \bfs(\Z) ).
  \]
\end{proposition}
\begin{proof} If $\Phi \in \Aut(\bfs)(\Q)$, then $\Phi$ is a
  $\Q$-defined automorphism of $\bfs$. Arithmetic groups are mapped to
  arithmetic groups under $\Q$-defined isomorphism of algebraic groups
  (see e.g.\ \cite[10.14]{raghunathan}), so $\Phi$ induces a
  commensuration of $\bfs(\Z)$.  Because $\bfs(\Z)$ is Zariski-dense
  in $\bfs$ by Theorem \ref{boreldensity}, the induced map $\Xi :
  \Aut(\bfs)(\Q) \to \Comm( \bfs(\Z) )$ is injective.
\end{proof}

The following consequence of Mostow--Prasad--Margulis rigidity is
likely known to experts. We include a proof, having found no reference
in the literature, using the techniques of
\cite{grunewaldplatonovrigidity}.
\begin{theorem} \label{semisimplecomm} Let $\bfs$ be a $\Q$-defined,
  connected, simply-connected, semisimple algebraic group without
  $\Q$-compact factors. Suppose that if $F$ is a factor of
  $\bfs(\R)^0$ locally isomorphic to $\PSL_2(\R)$ then $\bfs(\Z)$
  projects to a non-discrete subgroup of $F$. Then the inclusion
  \[
  \Xi : \Aut(\bfs) (\Q) \to \Comm( \bfs(\Z) )
  \]
  is an isomorphism.
\end{theorem}
\begin{proof}
  Let $\bfs_1, \dotsc, \bfs_n$ be the $\Q$-simple factors of
  $\bfs$, so that 
  \[
  \bfs = \bfs_1 \cdot \bfs_2 \cdot \dotsb \cdot \bfs_{n-1} \cdot \bfs_n.
  \]
  For each $j$, let $\pi_j : \bfs \to \bfs_j$ be the canonical
  projection.  

  Suppose $[\phi] \in \Comm(\bfs(\Z))$. Without loss of generality, we
  may assume that $\phi : \Gamma_1 \to \Gamma_2$ is a partial
  isomorphism of $\bfs(\Z)$ where
  \[
  \Gamma_1 = (\Gamma_1 \cap \bfs_1) \cdot (\Gamma_1 \cap \bfs_2) \cdot
  \dotsb \cdot (\Gamma_1 \cap \bfs_n).
  \]
  Let
  \[
  \Gamma_{1,i} = \Gamma_1 \cap \bfs_i \quad \text{ and } \quad
  \Gamma_1^i = \Gamma_{1,1}\cdot \dotsb \cdot \Gamma_{1,i-1} \cdot
  \Gamma_{1,i+1} \cdot \dotsb \cdot \Gamma_{1,n}. 
  \]
  Note that each $\Gamma_{1,i}$ is of finite index in $\bfs_i(\Z)$.
  
  Given any $i$, choose some $j$ such that $\pi_j( \phi(\Gamma_{1,i})
  )$ is infinite. Let $\bfa_1$ be the Zariski closure of $\pi_j(
  \phi( \Gamma_{1,i} ) )$ in $\bfs_j$, and $\bfa_2$ be the Zariski
  closure of $\pi_j( \phi( \Gamma_1^i ) )$ in $\bfs_j$. Replacing
  $\Gamma_1$ with a finite index subgroup if necessary, we may assume
  both $\bfa_1$ and $\bfa_2$ are connected. Then $\bfa_1$ commutes
  with $\bfa_2$ because $\Gamma_{1,i}$ commutes with
  $\Gamma_1^i$. Note that $\pi_j ( \phi(\Gamma_{1,i}) ) \cdot \pi_j(
  \phi( \Gamma_1^i ) )$ is commensurable with $\bfs_j(\Z)$, hence
  Zariski-dense in $\bfs_j$ by Theorem \ref{boreldensity}. Therefore
  $\bfa_1 \cdot \bfa_2 = \bfs_j$. Since $\pi_j( \phi( \Gamma_{1,i} )
  )$ is infinite and $\Q$-defined, and $\bfs_j$ is $\Q$-simple, it
  must be that $\bfa_1 = \bfs_j$. Since $\bfa_1$ commutes with
  $\bfa_2$ and $\bfa_2$ is connected, it follows that $\bfa_2$ is
  trivial. Therefore $\pi_j( \phi( \Gamma_1^i ) )$ must be trivial.

  It follows that, after replacing $\Gamma_1$ with a subgroup of
  finite index, for each $i$ there is exactly one $j$ so that $\pi_j
  ( \phi ( \Gamma_{1,i} ) )$ is nontrivial. Therefore for each $i$
  there is exactly one $j$ so that the image of $\Gamma_{1,i}$ under
  $\phi$ is a subgroup of $\bfs_j$ of finite index in $\bfs_j(\Z)$. It
  follows from Theorem \ref{semisimplerigidity} that
  $\restr{\phi}{\Gamma_{1,i}}$ virtually extends to an isomorphism
  $\Phi_i : \bfs_i \to \bfs_j$. The map $\Phi : \bfs \to \bfs$ defined
  by $\restr{\Phi}{\bfs_i} = \Phi_i$ is a $\Q$-defined automorphism
  virtually extending $\phi$, and so $\Xi$ is surjective.
\end{proof}

\subsection{More general lattices in semisimple groups}

A lattice $\Gamma$ in a connected semisimple Lie group $S$ with finite
center is {\em irreducible} if the projection of $\Gamma$ to $S/N$ is
dense for every nontrivial connected normal subgroup $N \leq S$.  Let
$\Gamma \leq S$ be an irreducible lattice in a connected semisimple
Lie group with trivial center and no compact factors. The relative
commensurator $\Comm_S(\Gamma)$ satisfies a dichotomy (see
\cite{zimmer}): either $\Comm_S(\Gamma)$ contains $\Gamma$ as a
subgroup of finite index, or $\Comm_S(\Gamma)$ is dense in $S$. In
fact, it is a celebrated theorem of Margulis that this is precisely
the dichotomy of arithmeticity versus non-arithmeticity.
\begin{theorem} [Margulis, see \cite{zimmer},
  \cite{margulis}] \label{commdichotomy} Let $\Gamma \leq S$ be an
  irreducible lattice in a connected semisimple Lie group with trivial
  center and no compact factors. Then $\Comm_S(\Gamma)$ is dense in
  $S$ if and only if $\Gamma$ is arithmetic.
\end{theorem}

We summarize the above results:
\begin{theorem} \label{semicomm} Let $\Gamma$ be an irreducible
  lattice in a noncompact connected semisimple Lie group $S$. Assume
  that $S$ is not locally isomorphic to $\PSL_2(\R)$. One of the
  following holds:
  \begin{enumerate}
  \item $\Gamma$ is arithmetic and there is a $\Q$-defined, connected,
    simply-connected, $\Q$-simple, semisimple algebraic group $\bfs$
    so that
    \[
    \Comm(\Gamma) \cong \Aut(\bfs)(\Q).
    \]
    Moreover, the group $\Aut(\Gamma)$ is commensurable with
    $\Aut(\bfs)(\Z)$. 
  \item $\Gamma$ is not arithmetic and $\Comm(\Gamma) \doteq \Gamma$.
  \end{enumerate}
\end{theorem}

\begin{proof}
  Suppose $\Gamma$ is arithmetic. Then there is a $\Q$-defined,
  connected, simply-connected, semisimple algebraic group $\bfs$
  without $\Q$-compact factors so that $\Gamma \doteq \bfs(\Z)$. Since
  $\Gamma$ is irreducible in $S$, the group $\bfs$ is $\Q$-simple. The
  isomorphism $\Comm(\Gamma) \cong \Aut(\bfs)(\Q)$ follows from
  Theorem \ref{semisimplecomm}. Since $\Aut(\Gamma)$ is commensurable
  with $\Gamma$ and $\Gamma$ is commensurable with $\bfs(\Z)$, the
  result follows since $\bfs(\Z)$ is commensurable with
  $\Aut(\bfs)(\Z)$.

  Now suppose $\Gamma$ is not arithmetic. Let $S' = S / Z(S)$ and $\pi
  : S \to S'$ the canonical projection. There is a finite index
  subgroup of $\Gamma$ taken faithfully to a lattice $\Gamma' \leq
  S'$. Let $N$ be the maximal compact factor of $S'$ and $S'' = S' /
  N$. Then $\Gamma'$ contains a finite index subgroup $\Gamma''$
  mapping isomorphically to a lattice $\Gamma'' \leq S''$. By
  Mostow--Prasad--Margulis rigidity (cf. \cite{mostowrigidity}),
  every commensuration of $\Gamma''$ extends to an automorphism of
  $S''$. Since $[ \Aut(S'') : \Inn(S'') ] < \infty$, where $\Inn(S'')$
  is the group of inner automorphisms of $S''$, it follows that
  $[\Comm(\Gamma'') : \Comm_{S''}(\Gamma'')] < \infty$, and hence $[
  \Comm(\Gamma'') : \Gamma''] < \infty$ by Theorem
  \ref{commdichotomy}. Since $\Gamma''$ is of finite index in
  $\Gamma$, the result follows.
\end{proof}

The case that $S = \PSL_2(\R)$ is dramatically different.
\begin{proposition} \label{psl2comm}
  Suppose $S$ is locally isomorphic to $\PSL_2(\R)$ and $\Gamma \leq
  S$ is a lattice. Then there is no faithful embedding $\Comm(\Gamma)
  \to \GL_N(\C)$ for any $N$.
\end{proposition}
\begin{proof}
  $\Gamma$ is either virtually free or virtually the fundamental group
  of a closed surface. All finitely generated free groups are
  abstractly commensurable to each other, as are all closed surface
  groups. Therefore we have that $\Comm(\Gamma)$ is isomorphic either
  to $\Comm(F_2)$ or to $\Comm( \pi_1(\Sigma_2) )$, where $F_n$ is the
  free group on $n$ letters and $\Sigma_g$ is a closed surface of
  genus $g$.

  A group $G$ has the {\em unique root property} if $x^k = y^k$
  implies $x=y$ for all $x,y\in G$ and nonzero $k$. If $G$ has the
  unique root property and $H\leq G$ is a finite index subgroup, then
  the natural map $\Aut(H) \to \Comm(G)$ is faithful (see
  \cite{odden}). It is easy to see that free groups and closed surface
  groups have the unique root property. Therefore $\Aut(F_n) \leq
  \Comm(F_2)$ for all $n\geq 2$, and $\Aut( \pi_1( \Sigma_g ) ) \leq
  \Comm( \pi_1( \Sigma_2) )$ for all $g\geq 2$.

  In \cite{formanekprocesi} it is shown that $\Aut(F_n)$ is not linear
  for any $n\geq 3$. Therefore $\Comm(F_2)$ cannot be linear. On the
  other hand, the proof of \cite[1.6]{farblubotzkyminsky} shows that
  for each $N$ there is some $g_0$ so that if $g\geq g_0$ then
  $\Mod^{\pm}( \Sigma_{g,1} )$, the extended mapping class group of
  the punctured surface of genus $g$, has no faithful complex linear
  representation of dimension less than or equal to $N$. Since
  $\Mod^{\pm}( \Sigma_{g,1} ) \cong \Aut( \pi_1( \Sigma_g ) )$, it
  follows that $\Comm( \pi_1( \Sigma_2) )$ is not linear.
\end{proof}

Nonarithmetic irreducible lattices can occur only in groups isogenous
to $\SO(1,n)$ or $\SU(1,n)$ up to compact factors. We will use this
fact in \S\ref{gensection}.

\begin{theorem}[see
  \cite{margulis},\cite{gromovschoen}] \label{arithmeticitytheorem}
  Let $S$ be a connected semisimple Lie group with trivial center and
  no compact factors. Suppose either $S = \Sp(1,n)$ for $n\geq 2$, or
  $S = F_4^{-20}$, or $\rrank(S) \geq 2$. Then every irreducible
  lattice in $S$ is arithmetic.
\end{theorem}

\subsection{Example: $\PGL_n(\Z)$}
\label{pslnsection}

Consider the algebraic group $\PGL_n$ for $n\geq 3$. The group
$\PGL_n(\R)^0$ is a semisimple Lie group, containing $\PGL_n(\Z) \cap
\PGL_n(\R)^0$ as a lattice.  By Theorem \ref{semisimplecomm} we have
\[ \Comm( \PGL_n(\Z) ) \cong \Aut( \PGL_n )(\Q).\]
Let $\tau : \PGL_n \to \PGL_n$ be the automorphism given by $\tau(A) =
(A^{-1})^t$. Then $\PGL_n$ acts on itself faithfully by conjugation,
and there is a decomposition
\[ \Aut(\PGL_n) = \PGL_n \rtimes \cyc{\tau}.\] Since $\tau$ preserves
$\PGL_n(\Z)$, there is an isomorphism
\begin{equation} \label{commpgl}
  \Comm( \PGL_n(\Z) ) \cong \PGL_n(\Q) \rtimes \cyc{\tau}.
\end{equation}

\begin{remark} Note that $\PSL_n(\R) = \PGL_n(\R)^0$ and $\PSL_n(\Z)
  \doteq \PGL_n(\Z)$, so it follows from equation (\ref{commpgl}) the
  above that
  \[ \Comm( \PSL_n(\Z) ) \doteq \PGL_n(\Q).\] In particular, $\Comm(
  \PSL_n(\Z) )$ is {\em not} commensurable with the group 
  \[ \PSL_n(\Q) = \SL_n(\Q) / Z( \SL_n(\Q) ).
  \]
  To
  understand this precisely, consider the $\Q$-defined surjection of
  algebraic groups $\pi : \SL_n \to \PGL_n$. The kernel of $\pi$ is
  isomorphic to the multiplicative group of order $n$, denoted
  $\mu_n$. By definition, $\PSL_n(\Q) = \pi( \SL_n(\Q) )$. As in
  \cite[2.2.3]{platonovrapinchuk}, the exact sequence of $\Q$-defined
  algebraic groups
  \[
  1 \to \mu_n \to \SL_n \to \PGL_n \to 1 
  \]
  gives rise to a long exact sequence of cohomology groups
  \[
  1 \to \mu_n(\Q) \to \SL_n(\Q) \to \PGL_n(\Q) \to H^1( \overline{\Q} /
  \Q, \mu_n ) \to 1.
  \]
  There is an isomorphism $H^1( \overline{\Q} / \Q , \mu_n ) \cong \Q^*
  / (\Q^*)^n$. This is infinitely generated for $n \geq 2$, hence
  $[\PGL_n(\Q) : \PSL_n(\Q)] = \infty$.
\end{remark}

\section{Commensurations of general lattices} 
\label{gensection}
Suppose $\Gamma$ is a lattice in a connected Lie group $G$ which is
not necessarily either solvable or semisimple. Our main result is:
\begin{shorttheoremstatement}
  Suppose $G$ is a connected, linear Lie group with connected,
  simply-connected solvable radical. Suppose $\Gamma \leq G$ is a
  lattice with the property that there is no surjection $\phi: G \to
  H$ to any group $H$ locally isomorphic to any $\SO(1,n)$ or
  $\SU(1,n)$ so that $\phi(\Gamma)$ is a lattice in $H$.  Then:
  \begin{enumerate}
  \item $\Gamma$ virtually embeds in a $\Q$-defined algebraic group
    $\bfg$ with Zariski-dense image so that every commensuration
    $[\phi] \in \Comm(\Gamma)$ induces a unique $\Q$-defined
    automorphism of $\bfg$ virtually extending $\phi$.
  \item There is a $\Q$-defined algebraic group
    $\mathcal{B}$ so that
    \[
    \Comm(\Gamma) \cong \mathcal{B}(\Q)
    \] 
    and the image of $\Aut(\Gamma)$ in $\mathcal{B}$ is commensurable
    with $\mathcal{B}(\Z)$.
  \end{enumerate}
\end{shorttheoremstatement}
The proof of Theorem \ref{shorttheorem} proceeds in four steps:
\begin{enumerate}
\item Construct the algebraic group $\bfg$, called the {\em virtual
    algebraic hull} of $\Gamma$, such that $\Gamma$ virtually embeds in
  $\bfg$ with Zariski-dense image.
\item Show that commensurations of $\Gamma$ induce $\Q$-defined
  automorphisms of $\bfg$.
\item Show that $\Aut(\bfg)$ has the structure of an algebraic group,
  and that $\Comm(\Gamma)$ is realized as the $\Q$-points of a
  $\Q$-defined subgroup of $\Aut(\bfg)$.
\item Show that the image of $\Aut(\Gamma)$ in $\Aut(\bfg)$ is
  commensurable with $\mathcal{B}(\Z)$.
\end{enumerate}
\begin{proof}[Proof of Theorem \ref{shorttheorem}:] Let $\Gamma$ be as
  in the theorem. Let $R$ be the solvable radical of $G$.
  
  \bold{Step 1:\ (Construction of virtual algebraic hull).} We will
  construct $\bfg$ as the semidirect product of a solvable group
  $\bfh$ with a semisimple group $\bfs$. Roughly speaking, $\bfh$ is
  the virtual algebraic hull of the ``solvable part'' of $\Gamma$,
  while $\bfs$ is a $\Q$-defined semisimple group without $\Q$-compact
  factors such that the ``semisimple part'' of $\Gamma$ is abstractly
  commensurable with $\bfs(\Z)$. To make this precise, we modify the
  Lie group $G$ and lattice $\Gamma$ as follows.

  Because $G$ is linear, there is a connected semisimple subgroup
  $S\leq G$ so that $G = R \rtimes S$. Let $\bfs'$ be a $\Q$-defined
  linear algebraic group so that $S = \bfs'(\R)^0$. There is a
  simply-connected algebraic group $\tilde \bfs'$ and a surjection
  $\pi: \tilde \bfs' \to \bfs'$ with finite central kernel. Let
  $\tilde S = \tilde \bfs'(\R)^0$. Then $\pi: \tilde S \to S$ is a
  finite covering map with central kernel. The lattice $\Gamma \leq
  R\rtimes S$ lifts to a lattice $\tilde \Gamma \leq R\rtimes \tilde
  S$, which is commensurable with $\Gamma$ by Proposition
  \ref{resfinite}. Replacing $R \rtimes S$ by $R \rtimes \tilde S$ and
  $\Gamma$ by $\tilde \Gamma$, we may assume that no finite cover of
  the semisimple quotient of $G$ has a linear representation,
  i.e. that $G$ is {\em algebraically simply-connected} (cf.\
  \cite[9.4]{witte}).

  Let $K$ be the maximal compact quotient of $S$ such that $\Gamma$
  projects to a finite subgroup of $K$. Because $G$ is algebraically
  simply-connected, $K$ may be identified with a subgroup of $S$, and
  there is a subgroup $S' \leq S$ so that $S = S' \times K$. Then
  $\Gamma \cap (R \rtimes S')$ is of finite index in $\Gamma$, so we
  may replace $S$ by $S'$ and assume that $\Gamma$ projects densely
  into the maximal compact factor of $S$. It follows by
  \cite[4.5]{starkovvanishing} that, passing to a finite index
  subgroup of $\Gamma$, we have chosen $S \leq G$ so that $\Gamma =
  (\Gamma \cap R)(\Gamma \cap S)$. Let $\Gamma_r = \Gamma\cap R$ and
  $\Gamma_s = \Gamma \cap S$. This makes precise our notions of
  ``solvable'' and ``semisimple'' parts of $\Gamma$.

  We now want to find a $\Q$-defined algebraic
  group $\bfs$ without $\Q$-compact factors so that $\Gamma_s$ is
  abstractly commensurable with $\bfs(\Z)$. Because $S$ is
  algebraically simply-connected, there is a decomposition $S = S_1
  \times \dotsb \times S_k$ so that $\Gamma_s$ virtually decomposes as
  $\Gamma_{s,1} \times \dotsb \times \Gamma_{s,k}$, where
  $\Gamma_{s,i} \leq S_i$ is an irreducible lattice for each
  $i$. Since each $\Gamma_{s,i}$ does not project to a lattice in
  $\SO(1,n)$ or $\SU(1,n)$, it follows from Theorem
  \ref{arithmeticitytheorem} that for each $i$ there is a connected
  $\Q$-defined semisimple algebraic group $\bfs_i$ and a surjection
  $\pi_i : \bfs_i(\R)^0 \to S_i$ with compact kernel so that $\pi_i(
  \bfs_i (\Z) \cap \bfs_i(\R)^0 )$ is commensurable with
  $\Gamma_{s,i}$. Set
  \[
  \bfs = \bfs_1 \times \dotsb \times \bfs_k \quad \text{ and }\quad
  \Gamma_s' = \prod_{i=1}^k \bfs_i(\Z) \cap \bfs_i(\R)^0.
  \]
  Each $\bfs_i$ is $\Q$-simple and $\bfs_i(\R)^0$ is not compact, so
  $\bfs$ is without $\Q$-compact factors.

  Our next goal is to define an action of $\bfs$ on the virtual algebraic
  hull of $\Gamma_r$.  To do this, we use the fact that the virtual
  algebraic hull of $\Gamma_r$ is a real algebraic hull for any
  unipotently connected, simply-connected solvable Lie group $R$
  containing $\Gamma_r$ as a Zariski-dense lattice. A classical
  construction may be used to produce a simply-connected solvable Lie
  group $R'$ so that $\Gamma_r$ is Zariski-dense in $R'$ and $R'$ is
  unipotently connected. To ensure that we can apply this construction
  while respecting the action of $S$, we present a proof based on
  ideas in \cite{bauesklopsch}.
  \begin{lemma}
    Suppose $G = R \rtimes S$ is a connected linear Lie group with $R$
    simply-connected solvable and $S$ semisimple. Let $\Gamma =
    (\Gamma \cap R) (\Gamma \cap S)$ be a lattice, and set $\Gamma_r =
    \Gamma \cap R$ and $\Gamma_s = \Gamma \cap S$. There is a finite
    index subgroup $\Gamma' \leq \Gamma$ of the form $\Gamma' =
    \Gamma_r' \rtimes \Gamma_s$ and a simply-connected solvable Lie
    group $R'$ so that $\Gamma'$ is a lattice in $R' \rtimes S$ with
    the property that $\Gamma_r'$ is Zariski-dense in $R'$ and $R'$ is
    unipotently connected.
  \end{lemma}
  \begin{proof}
    Let $\bfh_R$ be the real algebraic hull of $R$ and $\bfh_\Gamma$
    the virtual algebraic hull of $\Gamma_r$. There is a finite index
    characteristic subgroup $\Gamma_r' \leq \Gamma_r$ so that
    $\bfh_\Gamma$ is the algebraic hull of $\Gamma_r'$. By
    \cite[5.3]{bauesklopsch} we may moreover assume that there is some
    simply-connected solvable Lie group $R'$ that is unipotently
    connected and so that $\Gamma_r'$ is Zariski-dense in $R'$. The
    algebraic group $\bfh_\Gamma$ is a real algebraic hull for $R'$ by
    \cite[3.11]{bauesklopsch}. In particular, we identify $R'$ with a
    subgroup $R' \leq \bfh_\Gamma(\R)$ containing $\Gamma_r'$.

    By \cite[3.9]{bauesklopsch}, the inclusion $\Gamma_r' \leq R$
    extends to an $\R$-defined embedding $\bfh_\Gamma \to \bfh_R$. The
    action of $S$ on $R$ extends to an action of $S$ on $\bfh_R$ by
    $\R$-defined algebraic automorphisms.  Let $\Phi$ be an
    $\R$-defined automorphism of $\bfh_R$ induced by some $s\in
    S$. We would like to show that $\Phi$ preserves $R'$.

    Let $N$ be the maximal connected nilpotent normal subgroup of $R$,
    and let $\F$ denote the Zariski-closure of $\Fitt(\Gamma)$ in
    $\bfh_R$. We have $N \leq \F$ by a classical result of Mostow. It
    follows from \cite[3.3]{bauesklopsch} that $N \leq \bfh_R(\R)$ is
    normal.  Because $S$ is connected, the action of $S$ on $R / N$ is
    trivial by \cite[6.9]{bauesklopsch}. It follows that $\Phi(\F) =
    \F$.  By density of $R \leq \bfh_R$, we conclude that $\Phi$ is
    trivial on the quotient $\bfh_R / \F$.
    
    Let $N'$ be the maximal normal nilpotent subgroup of $R'$. Then
    $\F(\R) = N'$ in $\bfh_\Gamma$ because $R'$ is unipotently
    connected. It follows that $\Phi(R') \subseteq R' \F(\R) = R'$,
    and so $\Phi$ induces an automorphism of $R'$. This agrees with
    the given action of $\Gamma_s$ on $\Gamma_r'$, so we may form the
    semidirect product $G' = R' \rtimes S$ containing the lattice
    $\Gamma' = \Gamma_r' \rtimes \Gamma_s$.
  \end{proof}

  We may therefore assume that the radical $R$ of $G$ is unipotently
  connected and $\Gamma_r$ is Zariski-dense in $R$. Let $\bfh$ be the
  virtual algebraic hull of $\Gamma_r$. Because $R$ is unipotently
  connected and $\Gamma_r$ is Zariski-dense in $R$,
  \cite[3.11]{bauesklopsch} implies that $\bfh$ has the structure of a
  $\R$-defined connected algebraic hull of $R$. There is a
  representation $\rho : S \to \Aut_\R(\bfh)$ by the automorphism
  extension property of the algebraic hull. Because $\bfs$ is
  simply-connected, $\rho$ extends to an $\R$-defined representation
  $\rho : \bfs \to \Aut(\bfh)$ by Proposition
  \ref{repextension}. Since $\Gamma_s$ preserves $\Gamma_r$, we have
  that $\rho(\gamma)$ is $\Q$-defined for every $\gamma \in
  \Gamma_s$. Because $\bfs$ is without $\Q$-compact factors and
  connected, we know $\Gamma_s$ is Zariski-dense in $\bfs$ by Theorem
  \ref{boreldensity}. It follows from a standard fact, e.g.\
  \cite[I.0.11]{margulis}, that the representation $\rho : \bfs
  \to \Aut(\bfh)$ is $\Q$-defined.

  The definition of the variety structure on $\Aut(\bfh)$ implies that
  the action map $\alpha : \bfh \times \Aut(\bfh) \to \bfh$ is a
  $\Q$-defined map of varieties. It follows that the action map $\bfh
  \times \bfs \to \bfh$ is $\Q$-defined. The semidirect product of
  groups
  \begin{equation} \bfg = \bfh \rtimes \bfs \end{equation} therefore
  has the structure of a $\Q$-defined algebraic group. It is evident
  from the construction that $\Gamma$ embeds in $\bfg(\Q)$ as a
  Zariski-dense subgroup. This concludes the first step of the proof.

  \bold{Step 2:\ (Extension of commensurations).} We now construct a map
  \[
  \xi : \Comm(\Gamma) \to \Aut_\Q(\bfg).
  \]
  Let $\Lambda$ be a thickening of $\Gamma_r$ in $\bfh$ with nilpotent
  supplement $C$ and good unipotent shadow $\theta$, as in Proposition
  \ref{unipotentshadow}. The action of $\Gamma_s$ on $\Gamma_r$
  extends to an action on $\Lambda$. Then $\Lambda \rtimes \Gamma_s$
  is a Zariski-dense subgroup of $\bfg(\Q)$ containing $\Gamma$ as a
  finite index subgroup.

  \begin{lemma} \label{unipotentconj} Let $\U$ denote the unipotent
    radical of $\bfh$. Suppose $u \in \U(\Q)$. Then conjugation by $u$
    induces a commensuration of $\Gamma$.
  \end{lemma}
  \begin{proof}
    Suppose $u\in \U(\Q)$. Let $\F = \Fitt(\bfh)$. Conjugation by $u$
    induces two partial automorphisms: a partial automorphism
    $\phi_\theta : \theta_1 \to \theta_2$ of $\theta$, and a partial
    automorphism $\phi_R: \Lambda_1 \to \Lambda_2$ of $\Gamma_r$ by
    Theorem \ref{commhfdescription}. As in the proof of Theorem
    \ref{commhfdescription}, we may choose $\theta_1$, $\theta_2$,
    $\Lambda_1$, and $\Lambda_2$ so that $\theta_i \cap \F =
    \Fitt(\Lambda_i)$ for $i=1,2$. We want to find some finite index
    subgroup $\Gamma_s'' \leq \Gamma_s$ so that conjugation by $u$
    induces an isomorphism $\Lambda_1 \Gamma_s'' \to \Lambda_2
    \Gamma_s''$.

    Let $N$ be the maximal connected, normal, nilpotent subgroup of
    $R$. Because $S$ is connected, the action of $S$ on $R$ is trivial
    on $R / N$ (see \cite[6.9]{bauesklopsch}). Since we have assumed
    that $R$ is unipotently connected, $N$ is Zariski-dense in the
    Fitting subgroup $\F \leq \bfh$ by \cite[5.4]{bauesklopsch}, and
    so the induced action of $\Gamma_s$ on $\bfh$ is trivial on the
    quotient $\bfh / \F$. Therefore for any $s\in \Gamma_s$ we have
    \begin{equation} \label{innsemi} s u s^{-1} u^{-1} \in \F.
    \end{equation}
    Restricting our attention to $\Lambda$, we see that for any $s\in
    \Gamma_s$ and $c\in C$, there is some $f\in \Fitt(\Lambda)$ so
    that $s c s^{-1} = f c$. It follows that conjugation by $s\in
    \Gamma_s$ preserves $\theta$. Let $\Gamma_s' \leq \Gamma_s$ be a
    finite index subgroup normalizing both $\Lambda_1$ and
    $\Lambda_2$. Then $\Gamma_s'$ also normalizes both $\theta_1$ and
    $\theta_2$. By Lemma \ref{shadowaction}, there is a finite index
    subgroup $\Gamma_s'' \leq \Gamma_s'$ so that $u s u^{-1} s^{-1}
    \in \theta_1 \cap \theta_2$ for all $s\in \Gamma_s''$. Combining
    this with (\ref{innsemi}), for all $s\in \Gamma_s''$ we have
    \begin{equation} \label{innsemi2} 
      u s u^{-1} s^{-1} \in \Fitt(\Lambda_1) \cap \Fitt(\Lambda_2).
    \end{equation}
    The same arguments as in Claim 1 of the proof of Theorem
    \ref{commhfdescription} show that conjugation by $u$ induces a
    partial isomorphism $\Lambda_1 \Gamma_s'' \to \Lambda_2
    \Gamma_s''$ of $\Lambda \rtimes \Gamma_s$.
  \end{proof}
  
  \begin{proposition} \label{semipreserve} Every commensuration $[\phi]
    \in \Comm(\Gamma)$ induces a unique $\Q$-defined automorphism of
    $\bfg$ virtually extending $\phi$. Hence there is an injective
    homomorphism
    \[ \xi : \Comm(\Gamma) \to \Aut_\Q(\bfg).\]
  \end{proposition}
  \begin{proof}
    Suppose there are finite index subgroups $\Gamma_1$ and $\Gamma_2$
    of $\Lambda \rtimes \Gamma_s$ with $\phi : \Gamma_1 \to \Gamma_2$
    a partial automorphism representing $[\phi]$. Passing to a finite
    index subgroup so that $\Gamma_s \cap Z(S)$ is trivial, we may
    assume that $\Gamma_i \cap \bfh$ is the unique maximal normal
    solvable subgroup of $\Gamma_i$ for $i=1,2$ (cf. \cite[Lemma
    6]{prasad}). It follows that $\phi(\Gamma_1 \cap \bfh(\R) ) =
    \Gamma_2 \cap \bfh(\R)$, and so $\phi$ induces a commensuration
    $[\phi_R] \in \Comm(\Lambda)$ by Lemma \ref{inducedmaps}. It
    follows from Lemma \ref{solvcommembedding} that $\phi_R$
    extends to an automorphism $\Phi_R \in \Aut_\Q(\bfh)$.

    Now let $\bfl$ be the Zariski-closure of $\phi(\Gamma_1 \cap
    \Gamma_s)$ in $\bfg$. Then $\bfl$ is $\Q$-defined, and is
    semisimple by \cite[Theorem 2]{starkovvanishing}.  (Note that here
    we are using the assumption that $\Gamma_s$ does not surject to a
    lattice in any $\SU(1,n)$ or $\SO(1,n)$.) There is some $u\in
    \U(\Q)$ conjugating $\bfl$ into $\bfs$ by Theorem
    \ref{levimostow}. It follows from Lemma \ref{unipotentconj} that
    $\Inn_u \circ \phi$ virtually restricts to a partial automorphism
    $\phi_S : \Delta_1 \to \Delta_2$ of $\Gamma_s$. The partial
    automorphism $\phi_S$ virtually extends to a $\Q$-defined
    automorphism $\Phi_S \in \Aut_\Q(\bfs)$ by Theorem
    \ref{semisimplecomm}.

    Define an automorphism $\Phi \in \Aut(\bfg)$ by
    \[
    \Phi(r,s) = \Inn_{u^{-1}} \left( \Inn_u \circ \Phi_R(r), \Phi_S(s)
    \right).
    \]
    Then $\Phi$ virtually extends the partial automorphism
    $\phi$. This extension is unique up to choice of $u\in \U(\Q)$
    conjugating $\mathbf{L}$ to $\bfs$. However, any two such $u$
    differ by an element of $\U(\Q)$ centralized by $\bfs$, hence
    $\Phi$ is unique.
  \end{proof}

  \bold{Step 3:\ (Algebraic structure).} We now show that the image of
  $\xi : \Comm(\Gamma) \to \Aut_\Q(\bfg)$ has the structure of the
  $\Q$-rational points of a $\Q$-defined algebraic group. We first
  show that $\Aut(\bfg)$ in fact has the structure of a $\Q$-defined
  algebraic group.
  \begin{definition}
    A pair of automorphisms $(\Phi_R, \Phi_S) \in \Aut(\bfh) \times
    \Aut(\bfs)$ is {\em compatible} if there is some $\Phi \in
    \Aut(\bfg)$ preserving $\bfs$ with $\restr{\Phi}{\bfh} = \Phi_R$
    and $\restr{\Phi}{\bfs} = \Phi_S$. Let $C(\bfg) \subseteq
    \Aut(\bfh) \times \Aut(\bfs)$ be the set of compatible pairs of
    automorphisms.
  \end{definition}
  As both $\Aut(\bfh)$ and $\Aut(\bfs)$ have structures of
  $\Q$-defined algebraic groups, their product $\Aut(\bfh) \times
  \Aut(\bfs)$ is a $\Q$-defined algebraic group.
  \begin{lemma}
    $C(\bfg)$ is a $\Q$-defined subgroup of $\Aut(\bfh) \times
    \Aut(\bfs)$.
  \end{lemma}
  \begin{proof}
    Let $\rho : \bfs \to \Aut(\bfh)$ be the $\Q$-defined
    representation by conjugation. Any automorphism $\Phi \in
    \Aut(\bfg)$ preserving $\bfs$ must satisfy
    \[
    [ \Phi \circ \rho(s)] ( r ) = \Phi( s r s^{-1} ) = \Phi(s)
    \Phi(r) \Phi(s)^{-1} = [ \rho( \Phi(s) ) \circ \Phi] (r)
    \] 
    for all $r\in \bfh$ and all $s\in \bfs$. From this it is clear that
    any $(\Phi_R, \Phi_S) \in C(\bfg)$ satisfies
    \begin{equation} \label{compatible} \Phi_R \circ \rho(s) \circ
      \Phi_R^{-1} \circ \rho( \Phi_S(s) )^{-1} = \Id \in \Aut(\bfh)
    \end{equation}
    for all $s\in \bfs$. Conversely, suppose a pair $(\Phi_R, \Phi_S)
    \in \Aut(\bfh) \times \Aut(\bfs)$ satisfies (\ref{compatible}) for
    all $s\in \bfs$. Then the function $\Phi : \bfg \to \bfg$ defined by
    $\Phi(r,s) = \Phi_R(r) \Phi_S(s)$ is an automorphism of $\bfg$, and
    so $(\Phi_R, \Phi_S) \in C(\bfg)$. Thus $C(\bfg)$ is equal to the
    set of pairs $(\Phi_R, \Phi_S)$ satisfying (\ref{compatible}) for
    all $s\in \bfs(\Q)$. For a fixed element $s\in \bfs$, the solution set
    of equation (\ref{compatible}) is a $\Q$-defined closed subset of
    $\Aut( \bfh ) \times \Aut( \bfs )$. It follows that $C(\bfg)$ is a
    $\Q$-defined subgroup.
  \end{proof}

  \begin{lemma} \label{autbfgstructure}
    The map
    \begin{equation} \label{autgmap}
      \begin{split} \Theta : \U \rtimes C(\bfg)
        &\to \Aut(\bfg) \\ (u, \Phi_R, \Phi_S) &\mapsto \Inn_u \circ
        \Phi_R \circ \Phi_S 
      \end{split} 
    \end{equation}    
    is a surjective group homomorphism with $\Q$-defined unipotent
    kernel. Hence $\Aut(\bfg)$ has the structure of a $\Q$-defined
    algebraic group, such that
    \begin{equation} \label{autgqpoints} \Aut_\Q(\bfg) \cong
      \Aut(\bfg)(\Q) \cong \U(\Q) \rtimes C(\bfg)(\Q) / (\ker
      \Theta)(\Q).\end{equation}
  \end{lemma}
  \begin{proof}
    This follows from standard arguments. Compare to
    \S\ref{autstructure} and \cite[\S3.1]{bauesgrunewald}, for
    example.
  \end{proof}
  
  We will now show that the image of
  \[
  \xi : \Comm(\Gamma) \to \Aut(\bfg)
  \]
  is equal to the $\Q$-points of a $\Q$-defined subgroup of
  $\Aut(\bfg)$. Let $\mathcal{A}_{\Gamma_r} \leq \Aut(\bfh)$ be the
  $\Q$-defined subgroup such that $\mathcal{A}_{\Gamma_r}(\Q) \cong
  \Comm( \Gamma_r )$, as in Theorem \ref{introsolv}. Define
  \[
  \mathcal{B} = \left\{ \Phi \in \Aut(\bfg) \suchthat
    \restr{\Phi}{\bfh} \in \mathcal{A}_{\Gamma_r} \right\}.
  \]
  Then $\mathcal{B}$ is evidently a $\Q$-defined subgroup of
  $\Aut(\bfg)$. It is clear that $\xi( \Comm(\Gamma) ) \leq
  \mathcal{B}(\Q)$.
  \begin{proposition}
    The map $\xi : \Comm(\Gamma) \to \mathcal{B}(\Q)$ is an
    isomorphism.
  \end{proposition}
  \begin{proof}
    Clearly $\xi$ is injective. Suppose $\Phi \in \mathcal{B}(\Q)$. By
    Theorem \ref{levimostow} there is some $u\in \U(\Q)$ such that
    $\Inn_{u} \circ \Phi$ preserves $\bfs$. Since $\Inn_u \in
    \mathcal{A}_{\Gamma_r}$, it follows that $\Inn_{u} \circ \Phi \in
    \mathcal{B}(\Q)$. Therefore there are $\Phi_R \in
    \mathcal{A}_{\Gamma_r}(\Q)$ and $\Phi_S \in \Aut(\bfs)(\Q)$ such
    that $\Inn_u \circ \Phi = \Phi_R \circ \Phi_S$. 

    We have that $\Phi_R$ induces a partial automorphism $\phi_R :
    \Lambda_1 \to \Lambda_2$ of $\Lambda$ by Theorem \ref{introsolv},
    and $\Phi_S$ induces a partial automorphism $\phi_S : \Gamma_{s,1}
    \to \Gamma_{s,2}$ of $\Gamma_s$ by Proposition
    \ref{semicommqpointspart}. We may choose $\Lambda_1$ to be
    characteristic in $\Lambda$, and then choose $\Gamma_{s,2}$ to
    normalize $\Lambda_2 \leq \Lambda$. It follows that there is a
    well-defined isomorphism $\phi: \Lambda_1 \Gamma_{s,1} \to
    \Lambda_2 \Gamma_{s,2}$ defined by $\phi(r,s) = \Phi_R(r)
    \Phi_S(s)$, which clearly satisfies $\xi( [\phi] ) = \Phi_R \circ
    \Phi_S$. Since $\Inn_u \in \xi( \Comm(\Gamma) )$ by Lemma
    \ref{unipotentconj}, it follows that $\Phi \in \xi( \Comm(\Gamma)
    )$. 
  \end{proof}

  \bold{Step 4: ($\Aut(\Gamma)$ commensurable with
    $\mathcal{B}(\Z)$).}  It remains only to show that $\Aut(\Gamma)$
  is commensurable with $\mathcal{B}(\Z)$. For this, we first show
  that the element $u \in \U(\Q)$ arising in the proof of Proposition
  \ref{semipreserve} can be chosen in a controlled way. Given a vector
  space $V$ of finite dimension over a field of characteristic 0, we
  say that a subset $L \subseteq V$ is a {\em vector space lattice} if
  $L$ is a finitely generated $\Z$-submodule of $V(\Q)$ spanning $V$.
  \begin{lemma}\label{torusfixedpoints}
    Let $P$ be any group acting nontrivially and irreducibly on a
    vector space $V \cong \R^n$. Suppose $P$ preserves a vector space
    lattice $L' \subseteq V(\Q)$. Then there is a vector space lattice
    $L \subseteq V(\Q)$ such that if $v\in V(\Q)$ satisfies $v -
    p\cdot v \in L'$ for all $p\in P$ then $v\in L$.
  \end{lemma}
  \begin{proof}
    The action of $P$ descends to an action of $P$ on the torus $V /
    L'$. It suffices to show that this action has finitely many fixed
    points, as the fixed points of $V(\Q) / L'$ lift to the desired
    vector space lattice $L \subseteq V$. To see this, simply note
    that the fixed point set $X$ of the action of $P$ is a closed,
    hence compact, Lie subgroup of $V / L'$. The dimension of $X$ must
    be zero by the assumption that $P$ acts irreducibly and
    nontrivially on $V$. Therefore $X$ is finite.
  \end{proof}
  \begin{lemma} \label{uisnice} There is a subgroup $\Lambda \leq
    \U(\Q)$ commensurable with $\U(\Z)$ such that if $\phi \in
    \Aut(\Gamma)$ virtually extends to $\Phi \in \Aut(\bfg)$ then
    there is some $u \in \Lambda$ such that
    \[ (\Inn_u \circ \Phi) (\bfs) \subseteq \bfs. \]
  \end{lemma}
  \begin{proof}
    Let $\mathfrak{u}$ denote the Lie algebra of $\U$. The action of
    $\Gamma_s$ on $\U$ induces a linear action of $\Gamma_s$ on
    $\mathfrak{u}$. Let $\theta$ be a good unipotent shadow of
    $\Gamma_r$. Fix a vector space lattice $L' \subseteq
    \mathfrak{u}(\Q)$ containing $\log(\theta)$ preserved by the
    action of $\Gamma_s$ on $\mathfrak{u}$.

    Suppose $\phi \in \Aut(\Gamma)$ virtually extends to $\Phi \in
    \Aut(\bfg)$. By Theorem \ref{levimostow}, there is some $u\in
    \U(\Q)$ so that
    \begin{equation}\label{autconjugation}
      (\Inn_u \circ \Phi)(\bfs) \subseteq \bfs.
    \end{equation} 
    Define $\phi_1 : \Gamma_s \to \Gamma_r$ and $\phi_2 : \Gamma_s \to
    \Gamma_s$ by 
    \[
    \phi( 0 , \gamma_s ) = ( \phi_1(\gamma_s) , \phi_2(\gamma_s) ).
    \]
    Take any $\gamma_s \in \Gamma_s$. It follows from equation
    (\ref{autconjugation}) that $\phi_1( \gamma_s ) \in \U \cap
    \Gamma_r$, and so $\phi_1(\gamma_s) \in \theta$. From this we
    conclude that
    \[
      u ( \gamma_s \cdot u^{-1} ) \in \theta,
    \]
    and therefore
    \[
    \log(u) - \gamma_s \cdot \log(u) \in L'.
    \]

    Because $\bfs$ is semisimple, the action of $\Gamma_s$ on
    $\mathfrak{u}$ is completely reducible. Applying Lemma
    \ref{torusfixedpoints} to each irreducible component of this
    representation of $\Gamma_s$, we find a vector space lattice $L
    \subseteq \mathfrak{u}(\Q)$ with the property that any $u\in
    \U(\Q)$ satisfying equation (\ref{autconjugation}) satisfies
    $\log(u) \in L$. Let $\Lambda \leq \U(\Q)$ be any subgroup such
    that $\log(\Lambda)$ is a vector space lattice containing $L$ with
    finite index. Such a subgroup exists by the methods of
    \cite[\S6B]{segalbook}. The fact that $\Lambda$ is commensurable
    with $\U(\Z)$ is immediate from the fact that $\log(\Lambda)
    \subseteq \mathfrak{u}(\Q)$ is a vector space lattice.
  \end{proof}
  
  Now let 
  \[
  A_{\Lambda,\bfh} = \left\{ \Phi \in \ahf \suchthat \Phi( \Lambda )
    \subseteq \Lambda \right\}.
  \]
  Then $A_{\Lambda,\bfh}$ is commensurable with $\ahf(\Z)$ by
  \cite[8.1]{bauesgrunewald}, hence is commensurable with
  $\Aut(\Gamma_r)$. Define a $\Q$-defined subgroup of $C(\bfg)$ by
  \[
  C_\Gamma(\bfg) = \left\{ (\Phi_R, \Phi_S) \in C(\bfg) \suchthat
    \Phi_R \in \mathcal{A}_{\Gamma_r} \right\},
  \]
  and 
  \[
  A_\Lambda = \left \{ (\Phi_R,\Phi_S)\in C_\Gamma(\bfg) \suchthat
    \Phi_R \in A_{\Lambda,\bfh} \text{ and } \Phi_S(\Gamma_s) =
    \Gamma_s \right \}.
  \]
  Then $A_\Lambda$ is commensurable with $C_\Gamma(\bfg)(\Z)$.  Note
  that the map $\Theta$ of Lemma \ref{autbfgstructure} descends to a
  map
  \[
  \bar \Theta : \U \rtimes C_{\Gamma}(\bfg) \to \Aut(\bfg),
  \]
  and there is an isomorphism of algebraic groups
  \[
  \mathcal{B} \cong \U \rtimes C_{\Gamma}(\bfg) / \ker( \bar \Theta
  ). 
  \]
  
  Let 
  \[
  \Aut_\Lambda(\Gamma) = \left \{ \phi \in \Aut(\Gamma) \suchthat
    \restr{\phi}{\Gamma_r} \in A_{\Lambda,\bfh} \right\}.
  \]
  Note that $[\Aut(\Gamma) : \Aut_\Lambda(\Gamma)] < \infty$. By Lemma
  \ref{uisnice} there is a map
  \[
  \xi : \Aut_\Lambda(\Gamma) \to \Lambda \rtimes A_\Lambda / \ker(
  \bar \Theta ).
  \]
  This map is clearly injective, and the preceding discussion shows
  that its image is of finite index. Therefore the image of
  $\Aut(\Gamma)$ in $\mathcal{B}$ is commensurable with
  $\mathcal{B}(\Z)$. This completes the proof.
\end{proof}

\begin{remark}
  The assumption that the lattice $\Gamma$ is superrigid in $S$ cannot
  be removed from Theorem \ref{shorttheorem}. Consider for example $S
  = \SO(1,n)$ for $n\geq 2$ with a lattice $\Gamma \leq S$ such that
  $\Gamma / [\Gamma, \Gamma]$ is infinite. Let $\tau : \Gamma \to \Z$
  be any nontrivial homomorphism. Then $\phi_\tau : \Z\times \Gamma
  \to \Z\times \Gamma$ defined by
  \[
  \phi_\tau( t, \gamma) = ( t + \tau(\gamma) , \gamma )
  \]
  is an automorphism of $\Z\times \Gamma$, which is a lattice in $\R
  \times S$. However, $\phi_\tau$ neither is induced by conjugation by
  an element of $\Q \subseteq \R$ nor preserves $S$ in any sense, and
  $\phi_\tau$ cannot be extended to an automorphism of $\R \times S$.

  Automorphisms of the form $\phi_\tau$ as above are in one-to-one
  correspondence with elements of $H^1(\Gamma, \Z)$. If $\Delta \leq
  \Gamma$ is a finite index subgroup and $\sigma \in H^1(\Delta, \Z)$,
  then $\phi_\sigma$ defines a partial automorphism of $\Z \times
  \Gamma$. In this way we identify the inverse limit
  \[
  \mathcal{C} = \varprojlim \left \{ H^1(\Delta, \Z) \suchthat [\Gamma
    : \Delta] < \infty \right \}
  \]
  with a subgroup of $\Comm( \Z \times \Gamma)$. Nontrivial
  commensurations in $\mathcal{C}$ do not virtually extend to
  automorphisms of $\R \times S$. For any finite index subgroup
  $\Delta \leq \Gamma$, we may identify $H^1(\Delta, \Q)$ as a
  subgroup of $\mathcal{C}$. In this way, the virtual first rational
  Betti number of the semisimple quotient of a lattice may be seen as
  an obstruction to the realization of commensurations as
  automorphisms of an algebraic group.
\end{remark}

\bibliography{comm.bib}{}

\providecommand{\bysame}{\leavevmode\hbox to3em{\hrulefill}\thinspace}
\providecommand{\MR}{\relax\ifhmode\unskip\space\fi MR }
\providecommand{\MRhref}[2]{%
  \href{http://www.ams.org/mathscinet-getitem?mr=#1}{#2}
}
\providecommand{\href}[2]{#2}
\begin{thebibliography}{FLM01}

\bibitem[AB94]{acampoburger}
Norbert A'Campo and Marc Burger, \emph{R\'eseaux arithm\'etiques et
  commensurateur d'apr\`es {G}. {A}. {M}argulis}, Invent. Math. \textbf{116}
  (1994), no.~1-3, 1--25. \MR{1253187 (96a:22019)}

\bibitem[Avr14]{avramidi}
Grigori Avramidi, \emph{Smith theory, {$L^2$}-cohomology, isometries of locally
  symmetric manifolds, and moduli spaces of curves}, Duke Math. J. \textbf{163}
  (2014), no.~1, 1--34. \MR{3161310}

\bibitem[Bau04]{bauesinfrasolv}
Oliver Baues, \emph{Infra-solvmanifolds and rigidity of subgroups in solvable
  linear algebraic groups}, Topology \textbf{43} (2004), no.~4, 903--924.
  \MR{2061212 (2005c:57048)}

\bibitem[BG06]{bauesgrunewald}
Oliver Baues and Fritz Grunewald, \emph{Automorphism groups of
  polycyclic-by-finite groups and arithmetic groups}, Publ. Math. Inst. Hautes
  \'Etudes Sci. (2006), no.~104, 213--268. \MR{2264837 (2008c:20070)}

\bibitem[BK13]{bauesklopsch}
O.~Baues and B.~Klopsch, \emph{Deformations and rigidity of lattices in
  solvable {L}ie groups}, J. Topol. \textbf{6} (2013), no.~4, 823--856.
  \MR{3145141}

\bibitem[BN00]{biswasnag}
Indranil Biswas and Subhashis Nag, \emph{Limit constructions over {R}iemann
  surfaces and their parameter spaces, and the commensurability group actions},
  Selecta Math. (N.S.) \textbf{6} (2000), no.~2, 185--224. \MR{1816860
  (2002f:32026)}

\bibitem[Bor66]{boreldensity}
Armand Borel, \emph{Density and maximality of arithmetic subgroups}, J. Reine
  Angew. Math. \textbf{224} (1966), 78--89. \MR{0205999 (34 \#5824)}

\bibitem[Bor91]{borelbook}
\bysame, \emph{Linear algebraic groups}, second ed., Graduate Texts in
  Mathematics, vol. 126, Springer-Verlag, New York, 1991. \MR{1102012
  (92d:20001)}

\bibitem[dlH00]{delaharpe}
Pierre de~la Harpe, \emph{Topics in geometric group theory}, Chicago Lectures
  in Mathematics, University of Chicago Press, Chicago, IL, 2000. \MR{1786869
  (2001i:20081)}

\bibitem[FH07]{farbhandel}
Benson Farb and Michael Handel, \emph{Commensurations of {${\rm Out}({\rm
  F}_n)$}}, Publ. Math. Inst. Hautes \'Etudes Sci. (2007), no.~105, 1--48.
  \MR{2354204 (2008j:20102)}

\bibitem[FLM01]{farblubotzkyminsky}
Benson Farb, Alexander Lubotzky, and Yair Minsky, \emph{Rank-1 phenomena for
  mapping class groups}, Duke Math. J. \textbf{106} (2001), no.~3, 581--597.
  \MR{1813237 (2001k:20076)}

\bibitem[FP92]{formanekprocesi}
Edward Formanek and Claudio Procesi, \emph{The automorphism group of a free
  group is not linear}, J. Algebra \textbf{149} (1992), no.~2, 494--499.
  \MR{1172442 (93h:20038)}

\bibitem[FW05]{farbweinberger1}
Benson Farb and Shmuel Weinberger, \emph{Hidden symmetries and arithmetic
  manifolds}, Geometry, spectral theory, groups, and dynamics, Contemp. Math.,
  vol. 387, Amer. Math. Soc., Providence, RI, 2005, pp.~111--119. \MR{2179789
  (2006j:53075)}

\bibitem[FW08]{farbweinberger2}
\bysame, \emph{Isometries, rigidity and universal covers}, Ann. of Math. (2)
  \textbf{168} (2008), no.~3, 915--940. \MR{2456886 (2009k:53094)}

\bibitem[GP99a]{grunewaldplatonovrigidity}
Fritz Grunewald and Vladimir Platonov, \emph{Rigidity results for groups with
  radical cohomology of finite groups and arithmeticity problems}, Duke Math.
  J. \textbf{100} (1999), no.~2, 321--358. \MR{1722957 (2000j:11079)}

\bibitem[GP99b]{grunewaldplatonov}
\bysame, \emph{Solvable arithmetic groups and arithmeticity problems},
  Internat. J. Math. \textbf{10} (1999), no.~3, 327--366. \MR{1688145
  (2000d:20066)}

\bibitem[GS92]{gromovschoen}
Mikhail Gromov and Richard Schoen, \emph{Harmonic maps into singular spaces and
  {$p$}-adic superrigidity for lattices in groups of rank one}, Inst. Hautes
  \'Etudes Sci. Publ. Math. (1992), no.~76, 165--246. \MR{1215595 (94e:58032)}

\bibitem[Iva97]{ivanov}
Nikolai~V. Ivanov, \emph{Automorphism of complexes of curves and of
  {T}eichm\"uller spaces}, Internat. Math. Res. Notices (1997), no.~14,
  651--666. \MR{1460387 (98j:57023)}

\bibitem[LLR11]{leiningerlongreid}
Christopher Leininger, Darren~D. Long, and Alan~W. Reid, \emph{Commensurators
  of finitely generated nonfree {K}leinian groups}, Algebr. Geom. Topol.
  \textbf{11} (2011), no.~1, 605--624. \MR{2783240 (2012c:20139)}

\bibitem[LM06]{leiningermargalit}
Christopher~J. Leininger and Dan Margalit, \emph{Abstract commensurators of
  braid groups}, J. Algebra \textbf{299} (2006), no.~2, 447--455. \MR{2228321
  (2007b:20079)}

\bibitem[Mar91]{margulis}
G.~A. Margulis, \emph{Discrete subgroups of semisimple {L}ie groups},
  Ergebnisse der Mathematik und ihrer Grenzgebiete (3) [Results in Mathematics
  and Related Areas (3)], vol.~17, Springer-Verlag, Berlin, 1991. \MR{1090825
  (92h:22021)}

\bibitem[Mer70]{merz}
Ju.~I. Merzljakov, \emph{Integer representation of the holomorphs of polycyclic
  groups}, Algebra i Logika \textbf{9} (1970), 539--558. \MR{0280578 (43
  \#6298)}

\bibitem[Mos70]{mostowrep}
G.~D. Mostow, \emph{Representative functions on discrete groups and solvable
  arithmetic subgroups}, Amer. J. Math. \textbf{92} (1970), 1--32. \MR{0271267
  (42 \#6150)}

\bibitem[Mos73]{mostowrigidity}
\bysame, \emph{Strong rigidity of locally symmetric spaces}, Princeton
  University Press, Princeton, N.J., 1973, Annals of Mathematics Studies, No.
  78. \MR{0385004 (52 \#5874)}

\bibitem[NR92]{neumannreid}
Walter~D. Neumann and Alan~W. Reid, \emph{Arithmetic of hyperbolic manifolds},
  Topology '90 ({C}olumbus, {OH}, 1990), Ohio State Univ. Math. Res. Inst.
  Publ., vol.~1, de Gruyter, Berlin, 1992, pp.~273--310. \MR{1184416
  (94c:57024)}

\bibitem[Odd05]{odden}
Chris Odden, \emph{The baseleaf preserving mapping class group of the universal
  hyperbolic solenoid}, Trans. Amer. Math. Soc. \textbf{357} (2005), no.~5,
  1829--1858. \MR{2115078 (2005i:57018)}

\bibitem[PR94]{platonovrapinchuk}
Vladimir Platonov and Andrei Rapinchuk, \emph{Algebraic groups and number
  theory}, Pure and Applied Mathematics, vol. 139, Academic Press Inc., Boston,
  MA, 1994, Translated from the 1991 Russian original by Rachel Rowen.
  \MR{1278263 (95b:11039)}

\bibitem[Pra76]{prasad}
Gopal Prasad, \emph{Discrete subgroups isomorphic to lattices in {L}ie groups},
  Amer. J. Math. \textbf{98} (1976), no.~4, 853--863. \MR{0480866 (58 \#1015)}

\bibitem[Rag72]{raghunathan}
M.~S. Raghunathan, \emph{Discrete subgroups of {L}ie groups}, Springer-Verlag,
  New York, 1972, Ergebnisse der Mathematik und ihrer Grenzgebiete, Band 68.
  \MR{0507234 (58 \#22394a)}

\bibitem[Seg83]{segalbook}
Daniel Segal, \emph{Polycyclic groups}, Cambridge Tracts in Mathematics,
  vol.~82, Cambridge University Press, Cambridge, 1983. \MR{713786 (85h:20003)}

\bibitem[Sta02]{starkovvanishing}
A.~N. Starkov, \emph{Vanishing of the first cohomologies for lattices in {L}ie
  groups}, J. Lie Theory \textbf{12} (2002), no.~2, 449--460. \MR{1923777
  (2003e:22006)}

\bibitem[Tit74]{tits}
Jacques Tits, \emph{Buildings of spherical type and finite {BN}-pairs}, Lecture
  Notes in Mathematics, Vol. 386, Springer-Verlag, Berlin-New York, 1974.
  \MR{0470099 (57 \#9866)}

\bibitem[Weh94]{wehrfritz}
B.~A.~F. Wehrfritz, \emph{Two remarks on polycyclic groups}, Bull. London Math.
  Soc. \textbf{26} (1994), no.~6, 543--548. \MR{1315604 (96a:20047)}

\bibitem[Wit95]{witte}
Dave Witte, \emph{Superrigidity of lattices in solvable {L}ie groups}, Invent.
  Math. \textbf{122} (1995), no.~1, 147--193. \MR{1354957 (96k:22024)}

\bibitem[Zim84]{zimmer}
Robert~J. Zimmer, \emph{Ergodic theory and semisimple groups}, Monographs in
  Mathematics, vol.~81, Birkh\"auser Verlag, Basel, 1984. \MR{776417
  (86j:22014)}

\end{thebibliography}
\bibliographystyle{amsalpha}

\end{document}